\documentclass[a4paper,10pt]{amsart}

\usepackage[english]{babel}
\usepackage[utf8]{inputenc}
\usepackage{lmodern}
\usepackage[T1]{fontenc}

\usepackage{amsmath,amsfonts,amssymb,amsthm}
\usepackage[all]{xy}

\usepackage[dvipsnames]{xcolor}
\usepackage{color}
\usepackage{graphicx}
\usepackage[colorlinks=true]{hyperref}
\usepackage{microtype}

\newcommand\myshade{100}
\colorlet{mylinkcolor}{NavyBlue}
\colorlet{mycitecolor}{NavyBlue}
\colorlet{myurlcolor}{NavyBlue}
\hypersetup{
  linkcolor  = mylinkcolor!\myshade!black,
  citecolor  = mycitecolor!\myshade!black,
  urlcolor   = myurlcolor!\myshade!black,
  colorlinks = true,
}

\usepackage{enumitem} 

\usepackage{todonotes}
\usepackage{amsmath}

\newtheorem{proposition}{Proposition}[section]
\newtheorem{lemma}[proposition]{Lemma}

\newtheorem{theorem}[proposition]{Theorem}
\newtheorem{corollary}[proposition]{Corollary}

\theoremstyle{definition}
\newtheorem{definition}[proposition]{Definition}
\newtheorem{example}[proposition]{Example}

\theoremstyle{remark}
\newtheorem{remarkx}[proposition]{Remark}
\newenvironment{remark}
  {\pushQED{\qed}\remarkx}
  {\popQED\endremarkx}

\newcommand{\g}{\mathfrak{g}}

\renewcommand{\a}{\mathfrak{a}}
\newcommand{\m}{\mathfrak{m}}

\newcommand{\A}{\mathcal{A}}
\newcommand{\B}{\mathcal{B}}
\newcommand{\D}{\mathcal{D}}
\newcommand{\E}{\mathcal{E}}
\renewcommand{\L}{\mathcal{L}}
\newcommand{\M}{\mathcal{M}}
\newcommand{\N}{\mathcal{N}}

\newcommand{\KK}{\mathbb{K}}
\newcommand{\NN}{\mathbb{N}}

\newcommand{\RR}{\mathbb{R}}
\newcommand{\ZZ}{\mathbb{Z}}

\DeclareMathOperator{\At}{At}
\DeclareMathOperator{\End}{End}
\DeclareMathOperator{\ev}{ev}
\DeclareMathOperator{\Hom}{Hom}
\DeclareMathOperator{\id}{id}
\DeclareMathOperator{\pr}{pr}
\DeclareMathOperator{\sgn}{sgn}

\newcommand{\ol}{\overline}
\newcommand{\red}[1]{{#1}_{\textrm{red}}}  
\newcommand{\gm}[1]{\underline{#1}}  
\newcommand{\pmap}{\iota}  
\newcommand{\pb}[1]{\frac{\partial}{\partial b^{#1}}}  
\newcommand{\ka}{\kappa}  
\newcommand{\Clin}{C^\infty_{\text{lin}}}  
\newcommand{\DA}{D_A}
\newcommand{\DB}{D_B}
\newcommand{\XA}{X_A}
\newcommand{\XB}{X_B}

\title[Atiyah classes and dg-Lie algebroids for matched pairs]{Atiyah classes and dg-Lie algebroids \\ for matched pairs}
\date{\today}

\author{Panagiotis Batakidis}
\address{University of Cyprus, 1 Panepistimiou Avenue 
2109 Aglantzia, Nicosia, Cyprus}
\email{batakidis@gmail.com}
\author{Yannick Voglaire}
\address{University of Luxembourg, 6 Avenue de la Fonte, L-4364 Esch-sur-Alzette, Luxembourg}
\email{yannick.voglaire@gmail.com}

\thanks{We would like to extend our warm thanks to Ping Xu for suggesting the problem, and to Mathieu Stiénon and Ping Xu for many enlightening discussions.
We thank Tiffany Covolo, Nguyen Viet Dang, Owen Gwilliam, Benoît Jubin, Stephen Kwok, Camille Laurent-Gengoux, Rajan Mehta, François Petit, and Florian Schätz for discussions or correspondence about related subjects.
We are grateful to Damien Calaque who kindly commented on our first preprint and suggested us a relation to dDG-algebras.
P.B.\ is grateful to the Pennsylvania State University, where he was staying while most of this project was done, for the excellent working conditions.
He also thanks the Max Planck Institute for Mathematics in Bonn for its kind hospitality.
Y.V.\ is grateful to the Max Planck Institute for Mathematics in Bonn, where he was staying while most of this project was done, for the excellent working conditions.
}

\begin{document}


\begin{abstract}
For every Lie pair $(L,A)$ of algebroids we construct a dg-manifold structure on the $\ZZ$-graded manifold $\M=L[1]\oplus L/A$ such that the inclusion $\iota: A[1] \to \M$ and the projection $p:\M\to L[1]$ are morphisms of dg-manifolds.
The vertical tangent bundle $T^p\M$ then inherits a structure of dg-Lie algebroid over $\M$.
When the Lie pair comes from a matched pair of Lie algebroids, we show that the inclusion $\iota$ induces a quasi-isomorphism that sends the Atiyah class of this dg-Lie algebroid to the Atiyah class of the Lie pair.
We also show how (Atiyah classes of) Lie pairs and dg-Lie algebroids give rise to (Atiyah classes of) dDG-algebras.
\end{abstract}

\maketitle

\microtypesetup{protrusion=false} 
\tableofcontents 
\microtypesetup{protrusion=true} 


\section{Introduction}
\label{sec:introduction}

Atiyah classes form a bridge between complex geometry and Lie theory.
They were introduced by Atiyah \cite{atiyah_complex_1957} as the obstruction to the existence of a holomorphic connection on a complex manifold.
Much later, it was shown by Kapranov \cite{kapranov_rozansky-witten_1999} that the Atiyah class of a complex manifold $X$ endows the shifted holomorphic tangent bundle $T_X[-1]$ with a Lie algebra structure in the derived category  of coherent sheaves of $\mathcal{O}_X$-modules.
This structure plays an important role in the construction of Rozansky--Witten invariants \cite{kapranov_rozansky-witten_1999,kontsevich_rozansky-witten_1999}, which are parity-shifted analogues of the Chern--Simons invariants of three-manifolds.
In his work on the deformation quantization of Poisson manifolds \cite{kontsevich_deformation_2003}, Kontsevich provides another deep link between complex geometry and Lie theory, by relating a Lie algebraic analogue of the Todd class to the Duflo--Kirillov isomorphism.
The above-mentioned works of Kontsevich and Kapranov sparked a wealth of further investigation, which resulted in vast generalizations of the notion of Atiyah class \cite{calaque_hochschild_2010,chen_atiyah_2012,costello_geometric_2011,gwilliam_one-dimensional_2014,mehta_atiyah_2015,shoikhet_duflo_1998}.

The present paper intends to relate two such generalizations in the context of differential geometry.
The first, due to Chen--Stiénon--Xu \cite{chen_atiyah_2012}, is concerned with Lie pairs $(L,A)$, i.e.\ inclusions $A\subset L$ of Lie algebroids over the same base. 
The Atiyah class is then an element $\alpha_{(L,A)}$ in
\[ H^1(A, \Hom(L/A \otimes L/A, L/A)) . \]
The second, due to Mehta--Stiénon--Xu \cite{mehta_atiyah_2015}, is concerned with dg-Lie algebroids, i.e.\ graded Lie algebroids with a compatible homological vector field.
The Atiyah class of a dg-Lie algebroid $\A$ is an element $\alpha_\A$ in
\[ H^1(\Gamma(\Hom(\A\otimes \A, \A))) . \]

Our first main result realizes any Lie pair as a dg-Lie algebroid.
The homological vector fields involved are obtained as Fedosov differentials \cite{fedosov_simple_1994,dolgushev_covariant_2005}.
This result was independently proved by Stiénon--Xu \cite{stienon_fedosov_2016}.

\begin{theorem}
\label{thm:1st-main-theorem}
  Let $(L,A)$ be a Lie pair, $\phi:L/A\to L$ a splitting of the canonical projection, and $\nabla$ a torsion-free $L$-connection on $L/A$ extending the $A$-action.
  Then
  \[ \M = L[1] \oplus L/A \]
  has a dg-manifold structure such that the inclusion $\iota: A[1] \to \M$ and the projection $p:\M\to L[1]$ are morphisms of dg-manifolds, and the vertical tangent bundle
  \[ \D = T^p\M \to \M \]
  has a dg-Lie algebroid structure over $\M$.
  Both structures depend on $\phi$ and $\nabla$.
\end{theorem}

A particularly nice family of Lie pairs arises from the so-called matched pairs of Lie algebroids \cite{mokri_matched_1997}.
Two Lie algebroids $A$ and $B$ over the same base form a matched pair if their (vector bundle) direct sum $L=A\oplus B$ is endowed with a structure of Lie algebroid in which $A$ and $B$ are Lie subalgebroids. Given a matched pair as above, we get a Lie pair $(L,A)$ for which $L/A\cong B$. 

\begin{example}
  Let us give here a few examples of matched pairs of Lie algebroids for which Fedosov-type methods have been used (see the reference for each example below), and on which we largely draw in the context of this paper.
	\begin{enumerate}
	  \item $A=0$, $B=TM$ for a manifold $M$
  	  (Dolgushev 
	    \cite{dolgushev_covariant_2005}).
	  \item $A=M\times \g$, $B=TM$ for a manifold $M$ with a $\g$-action
	    (Dolgushev 
	    \cite{dolgushev_covariant_2005}).
	  \item $A=T^{0,1}X$, $B=T^{1,0}X$ for a complex manifold $X$ 
	    (Calaque--Rossi 
	    \cite{calaque_lectures_2011}).
	  \item $A=0$, $B=E$ for any Lie algebroid $E$ 
	    (Calaque--Dolgushev--Halbout 
	    \cite{calaque_formality_2007}).
	\end{enumerate}
\end{example}

Our second main result essentially says that, at least for matched pairs, the definition of Atiyah class for dg-Lie algebroids \cite{mehta_atiyah_2015} contains the definition of Atiyah class for Lie pairs \cite{chen_atiyah_2012}. 

\begin{theorem}
\label{thm:2nd-main-theorem}
  Let $L=A\oplus B$ be a matched pair of Lie algebroids, and $\nabla^B$ a torsion-free $B$-connection on $B$.
  Let $\D\to\M$ be the dg-Lie algebroid defined in Theorem~\ref{thm:1st-main-theorem}.

  Then the pullback by the inclusion $\iota:A[1]\to\M$ yields a quasi-isomorphism
  \[
    \iota^* : 
    (\Gamma(\Hom(\D \otimes \D, \D)),D)
    \to
    (\Gamma(\Lambda (A^\lor) \otimes \Hom(B\otimes B,B)),d_A)
  \]
  whose induced map in cohomology sends the Atiyah class $\alpha_\D$ of the dg-Lie algebroid $\D$ to the Atiyah class $\alpha_{(L,A)}$ of the Lie pair $(L,A)$.

  Here, $D$ is the dg-module structure on the space of sections of $\Hom(\D \otimes \D, \D)\to\M$ induced from the dg-manifold structure on $\D$, and $d_A$ is the Lie algebroid differential for the $\Hom(B\otimes B,B)$-valued $A$-cohomology.
\end{theorem}

In \cite[Section 8]{calaque_hochschild_2010}, Calaque and Van den Bergh introduce dDG-algebras and attach Atiyah classes to them.
We show that Lie pairs and dg-Lie algebroids both naturally define dDG-algebras through their Chevalley--Eilenberg algebras, and that the corresponding Atiyah classes in the sense of \cite{calaque_hochschild_2010} coincide with those of \cite{chen_atiyah_2012} and \cite{mehta_atiyah_2015}\footnote{In the case of matched pairs and dg-Lie algebroids, this was suggested to us by Damien Calaque after reading a first version of this manuscript.}.
In \cite[Section 8.4]{calaque_hochschild_2010}, the authors compute Atiyah classes using jet bundles in the case of a matched pair $A=0$, $B=E$ for a Lie algebroid $E$ and compute the Atiyah class of $E$ in terms of jet bundles. It would be interesting to extend our Theorem \ref{thm:2nd-main-theorem} to their setting and to extend their results to general Lie pairs.

The paper is structured as follows.
In Section \ref{sec:preliminaries}, we recall standard facts about Lie algebroids and matched pairs of Lie algebroids, we introduce dg-manifolds and dg-Lie algebroids, and describe the Atiyah classes we will be concerned with.
In Section~\ref{sec:dDG-algebras}, we relate Lie pairs, dg-Lie algebroids, and their Atiyah classes to dDG-algebras.
In Section~\ref{sec:DG-Lie-algebroids-from-Lie-pairs}, we prove Theorem~\ref{thm:1st-main-theorem}.
In Section~\ref{sec:quasi-isomorphism}, we prove Theorem~\ref{thm:2nd-main-theorem}.


\section{Preliminaries}
\label{sec:preliminaries}

In the whole paper, $\KK$ is either the field of real or complex numbers, 
$C^\infty(M)$ denotes the sheaf of smooth $\KK$-valued functions on a smooth manifold $M$,
and $T_\KK M$ denotes the tangent bundle $TM$ tensored with $M\times \KK$.
The sheaf of smooth sections of a vector bundle $E\to M$ is denoted by $\Gamma(M,E)$, or $\Gamma(E)$ when no confusion can arise.
The dual of a vector bundle $E$ is denoted by $E^\lor$.


\subsection{Lie algebroids and matched pairs}

A \emph{Lie algebroid} (over $\KK$) is a $\KK$-vector bundle $L\to M$ over a smooth manifold, together with a Lie bracket $[\cdot,\cdot]$ on the space $\Gamma(L)$ of sections of $L$, and a bundle map $\rho:L\to T_\KK M$ ---an \emph{anchor}--- satisfying the Leibniz rule
\[ [l,fl'] = f[l,l'] + \rho(l)(f) l' , \quad \forall l,l'\in\Gamma(L), f\in C^\infty(M) . \]

A \emph{(base-preserving) morphism of Lie algebroids} is a bundle map $L\to L'$ covering the identity $M\to M$, commuting with the anchors of $L$ and $L'$, and inducing a Lie algebra morphism $\Gamma(L)\to\Gamma(L')$ on sections.

For a Lie algebroid $L\to M$ and a vector bundle $E\to M$, an \emph{$L$-connection on $E$} is a $\KK$-linear map
\[ \nabla : \Gamma(L) \otimes_\KK \Gamma(E) \to \Gamma(E) : l\otimes e \mapsto \nabla_le \]
such that
\[
  \nabla_{fl}e = f\nabla_le , \qquad
  \nabla_lfe = f\nabla_le + \rho(l)(f)e,
\]
for all $l\in\Gamma(L)$, $e\in\Gamma(E)$, and $f\in C^\infty(M)$.
The \emph{curvature} of $\nabla$ is the bundle map $R^\nabla:L\wedge L \to \End(E)$ defined on sections by
\[ R^\nabla(l,l') = \nabla_l\nabla_{l'} - \nabla_{l'}\nabla_l - \nabla_{[l,l']}, \quad \forall l,l'\in\Gamma(L) . \]
The connection is called \emph{flat} if its curvature vanishes, and $E$ is then called an \emph{$L$-module}.

Any $L$-connection on $E$
\begin{enumerate}
  \item induces an $L$-connection on the dual vector bundle $E^\lor$ defined by
    $ \left< \nabla_l\epsilon,e\right> = \rho(l) (\left<\epsilon,e\right>) - \left< \epsilon, \nabla_le\right>$,
    for all $l\in\Gamma(L)$, $\epsilon\in\Gamma(E^\lor)$, and $e\in\Gamma(E)$;
  \item extends by derivations to the tensor algebra $T(E)$, the symmetric algebra $S(E)$, the symmetric algebra on the dual $S(E^\lor)$, the exterior algebra $\Lambda(E)$, etc.;
  \item when seen as a map 
  $\nabla:\Gamma(E)\to \Gamma(\Lambda^1 L^\lor \otimes E)$, 
  extends to an operator 
  $\nabla: \Gamma(\Lambda^\bullet L^\lor \otimes E) \to \Gamma(\Lambda^{\bullet+1} L^\lor \otimes E)$ 
  by the formula $\nabla(\omega\otimes e) = d\omega\otimes e + (-1)^n \omega \wedge \nabla e$ 
  for all $\omega\in\Lambda^n L^\lor$ 
  and $e\in \Gamma(E)$.\footnote{If $\nabla e = \sum_i \alpha_i \otimes e_i$, then $\omega\wedge \nabla e$ means $\sum_i (\omega\wedge \alpha_i) \otimes e_i$.}
\end{enumerate}

When $E$ is an $L$-module with flat connection $\nabla$, the degree 1 operator $d_E=\nabla$ on $\Gamma(\Lambda^\bullet L^\lor \otimes E)$ is a differential: $d_E\circ d_E = 0$. It is called the \emph{Lie algebroid differential} inducing the $E$-valued $L$-cohomology $H^\bullet(L,E)$. 
When $E$ is the trivial bundle $M\times \KK$, the anchor map defines a flat connection whose associated differential $d_L$ on $\Gamma(\Lambda^\bullet L^\lor)$ defines the Lie algebroid cohomology with trivial coefficients, $H^\bullet(L)$.

A \emph{Lie pair} \cite{chen_atiyah_2012} is a pair $(L,A)$ where $A$ is a Lie subalgebroid of a Lie algebroid $L$ over the same base.
Given a Lie pair $(L,A)$, the quotient $L/A$ is canonically an $A$-module with flat connection $\nabla^A_a\overline{l} = \overline{[a,l]}$. Here $l \mapsto \overline l$ is the projection $L\to L/A$.
An $L$-connection $\nabla$ on $L/A$ is said to \emph{extend the $A$-action} if $\nabla_a\overline l=\nabla^A_a\overline l$ for all $a\in\Gamma(A)$ and $l\in\Gamma(L)$.
The \emph{torsion} of an $L$-connection $\nabla$ on $L/A$ is the bundle map $T^\nabla:L\wedge L\to L/A$ defined on sections by
\begin{equation}
\label{eq:torsion}
  T^\nabla(l,l') = \nabla_l\overline{l'} - \nabla_{l'}\overline{l} - \overline{[l,l']} .
\end{equation}
The torsion tensor descends to a tensor $\overline T^\nabla:L/A\wedge L/A\to L/A$ if and only if the connection extends the $A$-action.
By the usual trick $\nabla\mapsto \nabla-\frac12 T^\nabla$, one can associate a torsion-free connection to any connection.

We will say that a Lie algebroid $L$ is \emph{formed from a matched pair} $(A,B)$ \cite{mokri_matched_1997} if $L$ is the direct sum (as a vector bundle) of two Lie subalgebroids $A$ and $B$. In that case, the Lie algebroid structure on $L$ induces an $A$-module structure $\nabla^A$ on $B$ and a $B$-module structure $\nabla^B$ on $A$ defined by
\begin{align*}
  \nabla^A_ab &= \pr_B[a,b], \\
  \nabla^B_ba &= \pr_A[b,a],
\end{align*}
for all $a\in\Gamma(A)$ and $b\in\Gamma(B)$. 
Here, $\pr_A,\pr_B$ denote the canonical projections from $L$ to $A$ and $B$, respectively.
Torsion-free $L$-connections $\nabla^{L}$ on $L/A\cong B$ extending the $A$-action are in bijection with torsion-free $B$-connections $\nabla^{B}$ on $B$.
The bijection is defined by $\nabla^{L} \longleftrightarrow \nabla^A + \nabla^{B}$.


\subsection{Differential graded manifolds}

In this section, we briefly explain the notions of differential graded manifold and differential graded Lie algebroid that we will use in this paper. 
There exist many variants of the notion of dg-manifold.
We follow \cite{mehta_atiyah_2015}\footnote{In addition to \cite{mehta_atiyah_2015}, we refer to \cite{mehta_algebroids_2009} for more details on the constructions presented in this section.}, whose definition might differ from others \cite{ciocan-fontanine_derived_2001, mehta_algebroids_2009, mehta_groupoids_2009, roytenberg_structure_2002, severa_some_2005, voronov_graded_2002} in at least two respects: the local model for the sheaf of functions on the underlying $\ZZ$-graded manifold is $C^\infty(U)\otimes \hat S(V^\lor)$, where $\hat S(V^\lor)$ are the \emph{formal polynomial functions} on a $\ZZ$-graded vector space $V$, and $V$ may have a \emph{non-trivial degree zero part}.
Two advantages of such requirements are the following.
First, working with formal polynomial functions ensures that the stalks are local rings.
Secondly, having degree zero formal variables allows to treat formal functions on a manifold simply as ``functions'' on a suitable graded manifold.

A \emph{$\ZZ$-graded manifold} $\M$ is a pair $(M,C^\infty(\M))$ composed of a smooth manifold $M$ and a sheaf $C^\infty(\M)$ of $\ZZ$-graded, graded-commutative algebras over $M$ such that there is a finite-dimensional $\ZZ$-graded $\KK$-vector space $V$ such that $C^\infty(\M)$ is locally isomorphic to $C^\infty(U)\otimes \hat S(V^\lor)$, $U\subset M$.
This means that, around each point $x\in M$, there exists an open neighborhood $U\subset M$ and an isomorphism $C^\infty(\M)|_U \to C^\infty(U)\otimes \hat S(V^\lor)$ of sheaves of graded algebras over $U$.

Above, a $\ZZ$-graded vector space $V=\oplus_{i\in\ZZ}V^i$ is called \emph{finite-dimensional} if $\sum_{i\in\ZZ}\dim V^i <\infty$.

All $\ZZ$-graded manifolds will be denoted by calligraphic letters $\M$, $\N$, \dots\ and the underlying smooth manifolds, called their \emph{body} or \emph{reduced manifold}, will be denoted by $\red{\M}$, $\red{\N}$, \dots. So, in the definition above, $\red{\M}=M$.

An \emph{open submanifold} of a $\ZZ$-graded manifold $\M$ is an open submanifold $U$ of $\red{\M}$ together with the restriction of the sheaf $C^\infty(\M)$ to $U$.

A \emph{morphism of $\ZZ$-graded manifolds} $\M\to \N$ is a pair $\phi=(f,\phi^*)$ where $f$ is a smooth map $\red{\M}\to \red{\N}$ and $\phi^*$ is a morphism of sheaves of $\ZZ$-graded algebras $C^\infty(\N)\to f_*C^\infty(\M)$. 

One can show  as in \cite[Section 6]{covolo_z2n_2014} that, although it is not apparent, there is a canonical augmentation map $C^\infty(\M) \to C^\infty(M)$, hence an inclusion $M\to \M$, that is compatible with all morphisms.

We will mostly consider, as in e.g.\ \cite[Appendix A]{cattaneo_relative_2007}, graded manifolds coming from ordinary graded vector bundles.
To an ordinary $\ZZ$-graded vector bundle $E=\oplus_{i\in\ZZ}E^i$ over $M$, we associate a graded manifold 
\begin{equation}
\label{eq:def-graded-manifold-from-graded-vector-bundle}
  \gm{E}=(M,C^\infty(\gm{E}))
\end{equation}
defined by $C^\infty(\gm{E})=\Gamma(M,\hat S(E^\lor))$ with augmentation given by evaluation at the zero section of $E$.
With the obvious map on morphisms, this yields a functor $E\mapsto \gm E$ from the category of graded vector bundles to that of graded manifolds.\footnote{In this language, the local model $(U,C^\infty(U)\otimes \hat S(V^\lor))$ in the definition of graded manifold is the graded manifold corresponding to the trivial vector bundle $U\times V$ over $U$.}

A \emph{dg-manifold} is a $\ZZ$-graded manifold $\M$ together with a homological vector field on $\M$, i.e.\ a derivation $Q$ of degree 1 of $C^\infty(\M)$ such that $Q\circ Q=0$.

A \emph{graded vector bundle} of rank $(k,\{k_i\}_{i\in\ZZ})$ over a graded manifold $\M$ is a graded manifold $\E$ with a morphism $\pi:\E\to\M$ and an atlas of local trivializations $\{\M_i,\phi_i\}_{i\in I}$ with, for each $i\in I$, $\M_i$ an open submanifold of $\M$ and $\phi_i:\pi^{-1}(\M_i)\to \M_i \times (\RR^k \times \KK^{\{k_i\}})$ an isomorphism of graded manifolds over a diffeomorphism $\red{\phi_i}:\red{(\pi^{-1}(\M_i))}\to \red{\M_i}\times \RR^k$, such that the transition functions are linear in the $\RR^k \times \KK^{\{k_i\}}$ coordinates.
Here, $\KK^{\{k_i\}}$ is the graded vector space $\oplus_{i\in\ZZ} \KK^{k_i}$ where each $\KK^{k_i}$ is in degree $i$.

It is clear from the definition of graded vector bundle that there is a well-defined subspace
\[ \Clin(\E) \subset C^\infty(\E) \]
of \emph{linear} functions on $\E$ (i.e.\ linear in the fibers of $\E\to\M$).

There is a degree-shift functor acting on the category of graded vector bundles.
First, define $\KK[-i]$ as the graded vector space with $\KK$ in degree $i$ and the zero vector space in all other degrees.
Given a graded manifold $\M$, define the (trivial) graded vector bundle $[i]_\M$ over $\M$ by $[i]_\M=\M\times(\RR\times \KK[i])$. 
So $[0]_\M$ is the unit for the monoidal category of graded vector bundles over $\M$ with the tensor product.
The \emph{$i$-th shift functor} $[i]$ on the category of graded vector bundles over $\M$ sends $\E$ to $\E[i]=\E\otimes [i]_\M$.

There is a canonical isomorphism $\phi_j:\Clin(\E)\to\Clin(\E[j])$ of degree $j$ of $C^\infty(\M)$-modules defined by $\phi_j(\epsilon) = \epsilon \chi_j$, where $\chi_j\in \Clin([j]_\M)$ is the linear coordinate on $\KK[j]$ with value 1 at 1.

For any graded vector bundle $\E$, $\Gamma^i(\E)=\Gamma^i(\M,\E)$ denotes the space of degree $i$ sections, i.e.\ of (degree-preserving) morphisms of graded manifolds $\sigma:\M\to\E[i]$ that are sections of the shifted projection $\pi[i]:\E[i]\to\M$. 

As explained e.g. in \cite[Remark 2.3]{mehta_algebroids_2009}, consistent choices of isomorphisms $\KK[i]\otimes\KK[j]\to\KK[i+j]$ can be made.
There result $C^\infty(\M)$-linear shift maps $\Gamma^j(\E)\to\Gamma^{j-k}(\E[k]):\sigma\mapsto\sigma_k$ defined by the commutative diagrams
\[
  \vcenter{\xymatrix{
    \M \ar[r]^-{\sigma^k} \ar[dr]_\sigma & \E[k][j-k] \ar[d]^\simeq \\
    & \E[j]
  }} .
\]

There is a non-degenerate $C^\infty(\M)$-bilinear pairing
\[ \left< \cdot, \cdot \right> : \Gamma(\E) \times \Clin(\E) \to C^\infty(\M) \]
defined by $\left< \sigma, \epsilon \right> = \sigma^* \circ \phi_{|\sigma|} (\epsilon)$.
For a section $\sigma\in\Gamma(\E)$, the contraction operator $i_\sigma$ is defined as the only degree $|\sigma|-1$ derivation of $C^\infty(\E[1])$ vanishing on $C^\infty(\M)$ and satisfying $i_\sigma \epsilon = \left< \sigma_1, \epsilon \right>$ for all $\epsilon\in\Clin(\E[1])$.

A (degree $k$) vector field on a graded manifold $\M$ is by definition a (degree $k$) derivation of $C^\infty(\M)$.
A \emph{linear} vector field on a graded vector bundle $\E\to\M$ is a vector field on $\E$ which preserves the subspace $\Clin(\E)$ of linear functions.
There is a correspondence between vector fields on $\M$ and sections of the tangent bundle $T\M$ \cite[Proposition 3.6]{mehta_algebroids_2009}.
Linear vector fields of degree $k$ on $\E$ correspond to sections of $T[k]\E\to\E$ which are bundle maps
\begin{equation*}
  \vcenter{\xymatrix{
    \E \ar[r] \ar[d] & T[k]\E \ar[d] \\
    \M \ar[r] & T[k]\M
  }} .
\end{equation*}
Linear vector fields can be shifted: any linear vector field $X$ of degree $k$ on $\E$ yields a linear vector field $X_j$ of degree $k$ on $\E[j]$ defined on linear functions by $$X_j = \phi_j \circ X \circ \phi_j^{-1}.$$

A \emph{dg-vector bundle} is a graded vector bundle $\pi:\E\to\M$ with a linear homological vector field $Q_\E$ on $\E$. 
This induces a homological vector field $Q_\M$ on $\M$ such that $Q_\E$ and $Q_\M$ are $\pi$-related. 
This also induces a degree 1 operator $Q_{\Gamma(\E)}$ on  $\Gamma(\E)$ which is $C^\infty(\M)$-linear and defined by
\begin{equation}
\label{eq:degree-one-operator-on-sections}
  \left<Q_{\Gamma(\E)}\sigma, \epsilon\right> = Q_\M \left<\sigma, \epsilon\right> - (-1)^{|\sigma|} \left<\sigma, Q_\E\epsilon \right>
\end{equation}
for all $\sigma\in\Gamma(\E)$ and $\epsilon\in \Clin(\E)$. 
Using shift maps, the last equation can also be read as
\begin{equation}
\label{eq:contraction-QGamma}
  i_{Q_{\Gamma(\E)}\sigma} = [(Q_\E)_1, i_\sigma] .
  \end{equation}

Since we will mostly consider that case, let us give an explicit form for the space of sections of graded vector bundles of the kind $\gm{E\oplus F}\to\gm{E}$, for two ordinary graded vector bundles $E,F$ over the same base.
\begin{lemma}
\label{lem:sections-of-graded-vector-bundles}
  Let $E,F\to M$ be two ordinary graded vector bundles. Then $\gm{E\oplus F}$ is a graded vector bundle over $\gm{E}$ and we have an isomorphism
  \begin{equation*}
  \label{eq:isomorphism-sections-of-graded-vector-bundle}
    \Gamma(\gm{E},\gm{E\oplus F}) \cong \Gamma(\hat S(E^\lor) \otimes F)
  \end{equation*}
  of graded $C^\infty(\gm{E})$-modules, where $\Gamma(\hat S(E^\lor) \otimes F)$ is graded by the total degree (i.e., the sum of the total degree in $\hat S(E^\lor)$ and the degree in $F$).
\end{lemma}

\begin{proof}
  The degree shift functor $[k]$ applied to the graded vector bundle $\gm{E\oplus F}\to\gm{E}$ yields simply $\gm{E\oplus F[k]}\to\gm{E}$.
  The space of functions on $\gm{E\oplus F[k]}$ is 
  $\Gamma(\hat S((E\oplus F[k])^\lor) \cong \Gamma(\hat S(E^\lor) \hat \otimes \hat S(F^\lor[-k]))$.
  
  To each section $\overline\sigma=\sum_i p_i\otimes f_i$ of total degree $k$ of $\hat S(E^\lor) \otimes F$ corresponds a section $\sigma=(\id,\sigma^*)$ of degree $k$ of the graded vector bundle $\gm{E\oplus F}\to\gm{E}$. 
  The latter section is the graded algebra morphism from $C^\infty(\gm{E\oplus F[k]})$ to $C^\infty(\gm{E})$ that is uniquely determined by the requirements $\sigma^*(1\otimes 1)=1$, $\sigma^*(\epsilon\otimes 1)=\epsilon$ and $\sigma^*(1\otimes \phi)= \sum_i \tilde\phi(f_i) p_i$, for all $\epsilon\in\Gamma(E^\lor)$ and $\phi\in\Gamma(F^\lor[-k])$. Here, $\tilde\phi$ denotes the element in $\Gamma(F^\lor)$ corresponding to $\phi$.
  
  Conversely, any degree $k$ section $\sigma=(\id,\sigma^*)$ of $\gm{E\oplus F}$, when restricted to \emph{linear} functions of the form $1\otimes \phi\in\Gamma(\hat S(E^\lor) \hat \otimes \hat S(F^\lor[-k]))$ with $\phi\in\Gamma(F^\lor[-k])$, yields a degree zero $C^\infty(M)$-linear map $\Gamma(F^\lor[-k])\to\Gamma(\hat S(E^\lor))$, hence a degree $k$ section $\overline\sigma$ of $\hat S(E^\lor)\otimes F^\lor$.
  
  It is easily checked that these assignments are inverse of each other and satisfy all the required properties.
\end{proof}

A \emph{graded Lie algebroid} is a graded vector bundle $\A\to\M$ with a degree zero bundle map $\rho:\A\to T\M$ (the anchor) and a degree zero Lie bracket $[\cdot,\cdot]$ on $\Gamma(\A)$ satisfying the Leibniz rule
\[ [X,fY] = \rho(X)(f) Y + (-1)^{|X| |f|} f [X,Y] \]
for all $X,Y\in\Gamma(\A)$ and $f\in C^\infty(\M)$.

As in the non-graded case, a graded Lie algebroid structure on a graded vector bundle $\A\to\M$ induces \cite[Theorem 4.6]{mehta_algebroids_2009} a (quadratic in the fiber coordinates) homological vector field $d_\A$ on $\A[1]$.
Locally, this vector field takes the form
\[ d_\A = \lambda^\alpha \rho^i_\alpha \frac{\partial}{\partial x^i} - (-1)^{|\lambda_\alpha|(|\lambda_\beta|+1)} \frac12 \lambda^\alpha \lambda^\beta C_{\alpha\beta}^\gamma \frac{\partial}{\partial \lambda^\gamma} . \]
Here, $x^i$ is a set of local coordinates on $\M$, $X_\alpha$ is a frame of local sections of $\A$ with degree $|X_\alpha|$, $\lambda^\alpha$ are the degree-shifted fiber coordinates dual to $X_\alpha$ (with $|\lambda^\alpha|=-|X_\alpha|+1$), $\rho^i_\alpha$ are the local functions defined by $\rho(X_\alpha)=\rho^i_\alpha \frac{\partial}{\partial x^i}$, and $C_{\alpha\beta}^\gamma$ are the structure functions defined by $[X_\alpha,X_\beta] = C_{\alpha\beta}^\gamma X_\gamma$.

A \emph{dg-Lie algebroid} is a dg-vector bundle $\A\to\M$ which is also a graded Lie algebroid such that the linear homological vector field $Q_\A$ and the Lie algebroid differential $d_\A$ satisfy $[(Q_\A)_1,d_\A]=0$.

\begin{example}
  \begin{enumerate}
    \item Any Lie algebroid $A$ trivially defines a dg-Lie algebroid $\gm A$ with $Q_{\gm A}=0$ .

    \item Given a graded Lie algebroid $\A$ over $\M$, $T[1]\A$ is a dg-Lie algebroid over $T[1]\M$ (see \cite[Section 5.2]{mehta_algebroids_2009}).
    
    \item Given a dg-manifold $\M$ with homological vector field $Q$, $T\M$ is dg-Lie algebroid over $\M$ such that the associated degree 1 operator $Q_{\Gamma(T\M)}$ is $L_Q$, the Lie derivative with respect to $Q$.
  \end{enumerate}
\end{example}

\begin{lemma}[{\cite[Lemma 1.7]{mehta_atiyah_2015}}]
\label{lem:integrable-distribution-dg-Lie-algebroid}
  Let $\M$ be a dg-manifold with homological vector field $Q$ and let $\D$ be an integrable distribution in $T\M$. If $\D$ is preserved by the Lie derivative $L_Q$, then $\D$ is a dg-Lie algebroid.
\end{lemma}

For a graded Lie algebroid $\L\to \M$ and a graded vector bundle $\E\to \M$, an \emph{$\L$-connection on $\E$} is a graded $\KK$-linear map
\[ \nabla : \Gamma(\L) \otimes_\KK \Gamma(\E) \to \Gamma(\E) : l\otimes e \mapsto \nabla_le \]
such that
\[
  \nabla_{fl}e = f\nabla_le , \qquad
  \nabla_lfe = (-1)^{|l||f|} f\nabla_le + \rho(l)(f)e,
\]
for all $l\in\Gamma(\L)$, $e\in\Gamma(\E)$, and $f\in C^\infty(\M)$.


\subsection{Atiyah classes}
\label{ssec:atiyah-classes}

In this section, we introduce the two notions of Atiyah class that we intend to relate in Section~\ref{sec:quasi-isomorphism} in the case of matched pairs of Lie algebroids. The first is the Atiyah class of an $A$-module relative to a Lie pair \cite{chen_atiyah_2012} and the second is the Atiyah class of a dg-vector bundle relative to a dg-Lie algebroid \cite{mehta_atiyah_2015}.

\subsubsection{Lie pairs}
\label{sssec:atiyah-class-lie-pairs}

Let $(L,A)$ be a Lie pair, and $E$ an $A$-module.
Let $\nabla$ be any $L$-connection on $E$ extending the $A$-action, i.e. such that $\nabla_ae$ coincides with the $A$-action $a\cdot e$ for all $a\in\Gamma(A)$ and $e\in\Gamma(E)$.
Let $R^\nabla:L\wedge L\to\End E$ be the curvature of $\nabla$, and define a map $\At^\nabla:A\to \Hom(L/A \otimes E, E)$ by
\begin{equation}
\label{eq:def-atiyah-cocycle-lie-pair}
  \At^\nabla(a)(\ol{l},e) = R^\nabla(a,l)e
\end{equation}
for all $a\in\Gamma(A)$, $l\in\Gamma(L)$, and $e\in\Gamma(E)$.

\begin{proposition}[{\cite{chen_atiyah_2012}}]
\label{prop:properties-atiyah-cocycle-Lie-pairs}
  \begin{enumerate}
    \item The map $\At^\nabla$ is well-defined and is a 1-cocycle for the $A$-cohomology with values in the $A$-module $\Hom(L/A \otimes E, E)$.
    \item The cohomology class $\alpha_{(L,A),E}$ of $\At^\nabla$ in $H^1(A,\Hom(L/A \otimes E, E))$ is independent of the choice of the connection $\nabla$ extending the $A$-action.
      It is called \emph{the Atiyah class of the $A$-module $E$ relative to the Lie pair $(L,A)$}.
  \end{enumerate}
\end{proposition}

\begin{definition}
  The \emph{Atiyah class $\alpha_{(L,A)}$ of a Lie pair} $(L,A)$ is $\alpha_{(L,A),L/A}$.
\end{definition}

\subsubsection{dg-Lie algebroids}
\label{sssec:Atiyah-class-dg-Lie-algebroids}

Let $(\A,Q_\A)$ be a dg-Lie algebroid and $(\E,Q_\E)$ a dg-vector bundle over the same base.
Let $\nabla$ be an $\A$-connection on $\E$.
Define a map $\At^\nabla:\A\otimes\E\to\E$ by
\begin{equation}
\label{eq:definition-Atiyah-cocycle-dgLiealgebroid}
  \At^\nabla(a,e) = Q_{\Gamma(\E)}(\nabla_ae) - \nabla_{Q_{\Gamma(\A)}(a)}e - (-1)^{|a|}\nabla_aQ_{\Gamma(\E)}(e)
\end{equation}
for all $a\in\Gamma(\A)$ and $e\in\Gamma(\E)$.

\begin{proposition}[{\cite{mehta_atiyah_2015}}]
\label{prop:properties-atiyah-cocycle-dg}
  \begin{enumerate}
    \item $\At^\nabla$ can be regarded as a degree $1$ section of $\Hom(\A \otimes \E, \E)$.
    \item $\At^\nabla$ is a cocycle: $Q(\At^\nabla) = 0$ for $Q=Q_{\Gamma(\Hom(\A \otimes \E, \E))}$.
    \item The cohomology class $\alpha_{\A,\E}$ of $\At^\nabla$ in $H^1(\Gamma(\Hom(\A \otimes \E, \E)),Q)$ is independent of the choice of the connection $\nabla$.
      It is called \emph{the Atiyah class of the dg-vector bundle $\E$ relative to the dg-Lie algebroid $\A$}.
  \end{enumerate}
\end{proposition}

\begin{definition}
  The \emph{Atiyah class $\alpha_\A$ of a dg-Lie algebroid} $\A$ is $\alpha_{\A,\A}$.
\end{definition}


\section{Relation to dDG-algebras}
\label{sec:dDG-algebras}

In this section, we show how the notions of Atiyah class for Lie pairs and for dg-Lie-algebroids (see Section \ref{ssec:atiyah-classes}) relate to the Atiyah class of dDG-algebras of \cite[Section 8]{calaque_hochschild_2010}. 
For completeness, we recall from \emph{op.\ cit.} the definitions of dDG-algebras, dDG-modules and the corresponding ``curvatures'', or Atiyah cocycles.

A (bigraded) DG-algebra on a manifold $M$ is a bigraded sheaf of algebras $\a$ on $M$ equipped with a derivation $\bar{d}_\a$ of bidegree $(1, 0)$ such that $\bar{d}_\a^2=0$. 
A dDG-algebra on $M$ is a bigraded sheaf of DG-algebras $\a$ on $M$ as above, equipped with a derivation $d_\a$ of bidegree $(0, 1)$ such that $\bar{d}_\a d_\a+ d_\a \bar{d}_\a= 0$.

If $\a$ is a DG-algebra, then a  DG-$\a$-module is a bigraded sheaf of
$\a$-modules $\m$ equipped with an additive map $\bar{d}_\m \colon \m\rightarrow \m$ of bidegree $(1, 0)$ such
that 
$\bar{d}_\m^2= 0$ and $\bar{d}_\m(\alpha m) = \bar{d}_\a(\alpha)m+(-1)^{|\alpha|}\alpha\bar{d}_\m(m)$
for all homogeneous sections $\alpha$ of $\a$ and $m$ of $\m$.

Finally, for a dDG-algebra $\A$, a connection on a DG-$\a$-module $\m$ is an additive map $d_\m\colon \m\rightarrow\m$ of bidegree $(0, 1)$ such that $d_\m(\alpha m) = d_\a(\alpha)m+(-1)^{|\alpha|}\alpha d_\m(m)$.
The corresponding curvature on the module $\m$ is defined as $R_\m=-(d_\m\bar{d}_\m + \bar{d}_\m d_\m)$.

The shift functors of bidegree $(1,0)$ and $(0,1)$ are denoted by ${}_-[1]$ and ${}_-(1)$, respectively.


\subsection{Lie pairs}
\label{ssec:Atiyah-class-dDG-algebra-of-a-Lie-pair}

Let $(L,A)$ be a Lie pair and let $j:L/A\to L$ be a splitting of the short exact sequence of vector bundles $$0\to A\to L\to L/A\to 0.$$ Denote by $B$ the image of $j$.
The Chevalley--Eilenberg algebra $\a = C(L) \cong C(A) \otimes C(B)$ is bigraded and, since $A$ is a subalgebroid, the degree 1 differential $d_L$ decomposes as a sum $$d_L=d^{1,0}_\a+d^{0,1}_\a+d^{-1,2}_\a$$ of three derivations of bidegree $(1,0)$, $(0,1)$, and $(-1,2)$, respectively.
Moreover, we have $(d^{1,0}_\a)^2=0$ and $d^{1,0}_\a d^{0,1}_\a+d^{0,1}_\a d^{1,0}_\a=0$.
Hence, $(C(L), d^{1,0}_\a, d^{0,1}_\a)$ is a dDG-algebra.

Let $(E,\nabla^A)$ be an $A$-module.
Consider $\m = \a \otimes \Gamma(E)$ with its left $\a$-action.
The $A$-module structure on $C(B)\otimes E$ gives a differential $d^{1,0}_\m$ (of bidegree $(1,0)$) on $\m \cong C(A) \otimes C(B) \otimes \Gamma(E)$ which satisfies
\[ d^{1,0}_\m(am) = d^{1,0}_\a(a)m + (-1)^{|a|} a d^{1,0}_\m(m)  \]
for all $a\in\a$ and $m\in\m$.
So $\m$ is a DG-module over $\a$.

Let $\nabla$ be an $L$-connection on $E$ extending the $A$-action.
It extends as a degree 1 endomorphism $d_\nabla$ of $\m$ which splits as a sum $d_\nabla = d^{1,0}_\m + d_\m^{0,1} + d_\m^{-1,2}$ of endomorphisms of bidegree $(1,0)$, $(0,1)$, and $(-1,2)$, respectively.
Here, as $\nabla$ extends the $A$-action, $d^{1,0}_\m$ is the differential defined in the foregoing paragraph.
Moreover, $d_\m^{0,1}$ satisfies
\[ d_\m^{0,1}(am) = d^{0,1}_\a(a)m + (-1)^{|a|} a d_\m^{0,1}(m) \]
i.e.\ it is a connection on the DG-$\a$-module $\m$.

The ``curvature'' of the connection $d_\m^{0,1}$ is defined in \cite[Lemma 8.2.4]{calaque_hochschild_2010} as the DG-$\a$-module map $R_\m:\m\to\m(1)[1]$ expressed by
$$ R_\m = -(d_\m^{0,1} d^{1,0}_\m + d^{1,0}_\m d_\m^{0,1}) . $$
It is a cocycle for the differential $[d^{1,0}_\m,{}\cdot{}]$ on $\Hom^{\bullet,\bullet}(\m,\m)$, and it is a coboundary if and only if there exists a connection with vanishing curvature (i.e. commuting with $d^{1,0}_\m$).
Moreover, its class $\alpha_\m\in H^1(\Hom^{\bullet,1}(\m,\m),[d^{1,0}_\m,{}\cdot{}])$ is independent of the chosen connection $d^{0,1}_\m$ (in our context, such connections always exist since we work over rings of smooth functions).

Since $E$ is an $A$-module, $\End(E)$ is also an $A$-module.
Consider the DG-$\a$-module isomorphism
\[ \mu: \a \otimes \Gamma(\End(E)) \to \Hom^{\bullet,\bullet}(\m,\m) \]
defined by $\mu(x)|_{\Gamma(E)} = x$ for all $x\in \a \otimes \Gamma(\End(E))$.
The Atiyah cocycle $\At^\nabla$ of $E$ is an element in $\a^{1,1}\otimes\Gamma(\End(E))$.
We now show that its image by $\mu$ is $R_\m$.

\begin{proposition}
The curvature of $\m$ is the only DG-$\a$-module map from $\m$ to $\m(1)[1]$ satisfying
$$i_a i_b R_\m(e)=\mathrm{At}^\nabla(a)(b,e)$$
for $a\in\Gamma(A), b\in\Gamma(B),e\in\Gamma(E)$.
Consequently, $\mu$ sends $\alpha_{(L,A),E}$ to $\alpha_\m$ in cohomology.
\end{proposition}

\begin{proof}
Let $a\in\Gamma(A), b\in\Gamma(B), e\in\Gamma(E)$. Recall that the Atiyah cocycle is 
\[\At^\nabla(a)(b,e)=R^\nabla(a,b)e=i_bi_ad_\nabla^2e.\] 
Using $d_\nabla=d^{1,0}_\m+d^{0,1}_\m+d^{-1,2}_\m$ and  $(d^{1,0}_\m)^2=0$, we get
\[i_bi_ad^2_\nabla e=i_ai_bR_\m (e).\]
\end{proof}


\subsection{dg-Lie algebroids}
\label{ssec:Atiyah-class-dDG-algebra-of-a-dg-Lie-algebroid}

Let $\M$ be a dg-manifold and $\A\to \M$ a dg-Lie algebroid. 

Recall that by definition $\A$ is equipped with a homological vector field $d_\A$ on $\A[1]$ and a linear homological vector field $Q_\A$ on $\A$ such that $[d_\A,(Q_{\A})_1]=0$. 

Consider the polynomial functions $\a=\Gamma(S(\A[1]^\lor))\subset C^\infty(\A[1])$ on $\A[1]$.
There is a grading on $\a$ induced by the one on functions on $\A[1]$.
This grading comes from a bigrading determined by the manifold degree of the dg-manifold $\A$ and the polynomial degree of the bundle $S(\A[1]^\lor)$. 
By construction, $(Q_\A)_1$ is of bidegree $(1,0)$ while $d_\A$ is of bidegree $(0,1)$. 
Hence $(\a,(Q_{\A})_1,  d_\A)$ is a dDG-algebra. 

Let $\E\to\M$ be a dg-vector bundle and set $\m=\a\otimes\Gamma(\E)$. 
The latter is an $\a$-module and has a bigrading inherited from the one of $\A$ and by declaring $\Gamma(\E)$ to lie in degrees $(\bullet,0)$.
Moreover, $\m$ has a differential $Q_\m$ of bidegree $(1,0)$ defined by 
$$Q_\m=(Q_\A)_1\otimes 1+1\otimes Q_{\Gamma(\E)},$$
where $Q_{\Gamma(\E)}$ is defined in (\ref{eq:degree-one-operator-on-sections}). 
Hence, $\m$ is a DG-module over $\a$. 

Let $\nabla$ be an $\A$-connection on $\E$ and denote by $d_{\A,\E}$ its extension to $\A$-forms. \
Since for all $a\in\a$ and $m\in\m$,
\[ d_{\A,\E}(am) = d_\A(a)m + (-1)^{|a|} a d_{\A,\E}(m) , \]
$d_{\A,\E}$ defines a connection on the DG-$\a$-module $\m$.

The curvature of the DG-$\a$-module $\m$ is 
$$R_\m=-[Q_\m,d_{\A,\E}].$$
Recall from \eqref{eq:definition-Atiyah-cocycle-dgLiealgebroid} the Atiyah cocycle  $\At^\nabla$ of the dg-vector bundle $\E$ relative to the dg-Lie algebroid $\A$.

As in the previous section, there is a DG-$\a$-module isomorphism
\[ \mu: \a \otimes \Gamma(\End(\E)) \to \Hom^{\bullet,\bullet}(\m,\m) \]
defined by $\mu(x)|_{\Gamma(\E)} = x$ for all $x\in \a \otimes \Gamma(\End(\E))$.
The Atiyah cocycle $\At^\nabla$ of $\E$ is an element of degree $(1,1)$ in $\a\otimes\Gamma(\End(\E))$.
We now show that its image by $\mu$ is $R_\m$.

\begin{proposition}
The curvature of $\m$ is the only DG-$\a$-module map from $\m$ to $\m(1)[1]$ satisfying
  \begin{equation}
    i_X R_\m (e) = (-1)^{|X|} \At^\nabla(X,e)
  \end{equation}
for all $X\in\Gamma(\A)$ and $e\in\Gamma(\E)$.
Consequently, $\mu$ sends $\alpha_{\A,\E}$ to $\alpha_\m$ in cohomology.
\end{proposition}
\begin{proof}
Using formula \eqref{eq:contraction-QGamma}, we get
\begin{align*}
  i_X [Q_\m,d_{\A,\E}](e) 
  &= i_X Q_\m \nabla e + i_X \nabla Q_{\Gamma(\E)} e \\
  &= (-1)^{|X|-1}(Q_\m i_X\nabla e-i_{Q_{\Gamma(\A)}X}\nabla e) + \nabla_X Q_{\Gamma(\E)} e \\
  &= (-1)^{|X|-1} Q_{\Gamma(\E)}\nabla_Xe-(-1)^{|X|-1}\nabla_{Q_{\Gamma(\A)}X} e + \nabla_X Q_{\Gamma(\E)}e \\
  &= -(-1)^{|X|}\At^\nabla(X,e).
\end{align*}
\end{proof}


\section{dg-Lie algebroids from Lie pairs}
\label{sec:DG-Lie-algebroids-from-Lie-pairs}

In this section, we prove our first main result, Theorem~\ref{thm:1st-main-theorem}.
The techniques used in this section and the next one closely follow Dolgushev's paper \cite{dolgushev_covariant_2005} and the subsequent works \cite{calaque_formality_2007}, \cite[Section 10]{calaque_lectures_2011}.


\subsection{Setup}
Let $(L,A)$ be a Lie pair, and let $\phi:L/A\to L$ be a splitting of the short exact sequence of vector bundles
\[ 0 \to A \to L \to L/A \to 0 . \]
Denote by $B$ the image of $\phi$ in $L$.

The inclusion $i:A\to L$ induces an inclusion of graded vector bundles $A[1]\to L[1]\to L[1]\oplus L/A$, where the second map is the inclusion as the first component.
Applying the functor \eqref{eq:def-graded-manifold-from-graded-vector-bundle} from graded vector bundles to graded manifolds, we get an inclusion of graded manifolds
\begin{equation*}
\label{eq:inclusion-A-LB}
  \pmap: \gm{A[1]} \to \gm{L[1]\oplus L/A}
\end{equation*}
over the identity from $M$ to $M$.

Using the splitting $\phi$, we may identify $L[1]\oplus L/A$ with $A[1]\oplus B[1] \oplus B$.
Hence, we have a projection of graded manifolds
\begin{equation*}
\label{eq:projection-ABB-A}
  \pi: \gm{A[1]\oplus B[1] \oplus B} \to \gm{A[1]}
\end{equation*}
over the identity from $M$ to $M$. 

Let us define 
\begin{align*}
  \M &= \gm{A[1]\oplus B[1] \oplus B} , & \A &= \gm{A[1]}.
\end{align*}
So we have
\begin{equation*}
  \pi: \M \rightleftarrows \A : \iota .
\end{equation*}

Let us pick a system of local coordinates 
\begin{equation}
\label{eq:coordinates}
  x^1,\dots,x^n,\alpha^{1},\dots,\alpha^{t},\beta^1,\dots,\beta^s,b^1,\dots,b^s
\end{equation}
on $\M$, where the $x^i$ are coordinates on $M$, the $\alpha^i$ (respectively, $b^i$) are linear coordinates on the fibers of $A[1]$ (respectively, $B$), and the $\beta^i$ are the $b^i$ seen as coordinates on $B[1]$ (so they have degree 1).
Let us denote by $\lambda^i$, $i=1,\dots,s+t$, the coordinates $\beta^1,\dots,\beta^s,\alpha^1,\dots,\alpha^t$ collectively.

The Euler vector field generating the homotheties of the vector bundle $B$ induces a degree 1 vector field $\delta$ on $\gm{B[1]\oplus B}$ and hence on $\M$.
It reads
\begin{equation*}
  \delta = {\id} \otimes {\beta^i} \otimes {\pb{i}}
\end{equation*}
or, more explicitly,
\begin{equation*}
  \delta(\alpha\otimes\beta\otimes b) = (-1)^{|\alpha|+|\beta|}  \alpha\otimes (\beta\wedge\beta^i)\otimes\pb{i}b,
\end{equation*}
for homogeneous elements $\alpha\in\Gamma(\Lambda(A^\lor))$, $\beta\in\Gamma(\Lambda(B^\lor))$, $b\in\Gamma(\hat S(B^\lor))$.
An easy check shows that $\delta$ is defined globally and that
\begin{equation*}
  \delta \circ \delta = 0 .
\end{equation*}

As in \cite{fedosov_simple_1994,dolgushev_covariant_2005,calaque_formality_2007}, there is a homotopy $\ka:C^\infty(\M) \to C^\infty(\M)$ vanishing on $\Gamma(\Lambda(A^\lor)\otimes \Lambda^0 B^\lor \otimes \hat S(B^\lor))$ and sending, for each $q\geq 1$ and $r\geq 0$, $\Gamma(\Lambda^\bullet A^\lor \otimes \Lambda^q B^\lor \otimes S^rB^\lor)$ to $\Gamma(\Lambda^\bullet A^\lor \otimes \Lambda^{q-1} B^\lor \otimes S^{r+1}B^\lor)$ by
\begin{equation*}
  \ka = \frac{1}{q+r} {\id} \otimes {i_{\frac{\partial}{\partial \beta^i}}} \otimes {b^i}
\end{equation*}
or, more explicitly,
\begin{equation*}
  \ka(\alpha\otimes\beta\otimes b) 
  = \frac{1}{q+r} (-1)^{|\alpha|} \alpha\otimes {i_{\frac{\partial}{\partial \beta^i}}\beta} \otimes b^ib ,
\end{equation*}
for homogeneous elements $\alpha\in\Gamma(\Lambda(A^\lor))$, $\beta\in\Gamma(\Lambda^q B^\lor)$, $b\in \Gamma(S^r B^\lor)$.

\begin{remark}
To shorten the notation, let us write for any vector bundle $F$ on $M$
\begin{align*}
  \Omega^\bullet(F) &= \Gamma(\Lambda^\bullet L^\lor \otimes F) , \\
  \Omega^{\bullet,\bullet}(F) &= \Gamma(\Lambda^\bullet A^\lor \otimes \Lambda^\bullet B^\lor \otimes F) ,
\end{align*}
and
\begin{align*}
  \B &= \hat S(B^\lor) \otimes \Hom(B\otimes B,B) , \\
  \B^{\geq i} &= \hat S^{\geq i}(B^\lor) \otimes \Hom(B\otimes B,B) , \\
  \B^i &= S^i(B^\lor) \otimes \Hom(B\otimes B,B) .
\end{align*}
So, for example, depending on the context we will use the following notation for the space of functions on $\M$,
\[ C^\infty(\M) = \Gamma(\Lambda(L^\lor)\otimes \hat S(B^\lor)) = \Omega^\bullet(\hat S(B^\lor)) , \]
this space being isomorphic to the space
\[ \Gamma(\Lambda(A^\lor) \otimes \Lambda(B^\lor) \otimes \hat S(B^\lor)) = \Omega^{\bullet,\bullet}(\hat S(B^\lor)) . \qedhere \]
\end{remark}

We now extend the definitions of $\delta$ and $\ka$ to some vector bundles over $\M$.
Specifically, consider the graded vector bundle
\begin{equation}
\label{eq:def-D}
  \D=\gm{L[1]\oplus B\oplus B} \to \M=\gm{L[1]\oplus B},
\end{equation}
whose sections are $\Gamma(\D)=\Gamma(\Lambda (L^\lor) \otimes \hat S(B^\lor) \otimes B)$.
We consider it as the kernel of the differential of the projection of graded manifolds $L[1]\oplus B\to L[1]$, hence as an integrable subbundle of $T\M$ and can thus see its sections as derivations of the algebra of functions on $\M$: elementary (local) sections $Y=\lambda \otimes b \otimes \pb{i} \in\Gamma(\D)$ act on elementary functions $f=\lambda' \otimes b' \in C^\infty(\M)$ by
$$ Y(f) = \lambda \wedge \lambda' \otimes b \pb{i} b' . $$
The operator $\delta$ 
then extends to $\Gamma(\D)$ by the Lie derivative,
\[ \delta(Y) = [\delta,Y] = \delta \circ Y - (-1)^{|Y|} Y \circ \delta . \]
We consider a second bundle obtained from $\D$, the bundle $\Hom(\D\otimes\D,\D)$ of graded vector bundle homomorphisms from $\D\otimes \D$ to $\D$, whose sections are
\[ \Gamma(\Hom(\D\otimes\D,\D)) = \Omega^\bullet(\B) = \Gamma(\Lambda(L^\lor) \otimes \hat S(B^\lor) \otimes \Hom(B\otimes B,B)) . \]
If $\phi\in\Gamma(\Hom(\D\otimes\D,\D))$, we define
\[ \delta(\phi) = \delta \circ \phi - (-1)^{|\phi|} \phi \circ (\delta\otimes 1 + 1 \otimes \delta) . \]

Since $\Gamma(\D)$ is generated by $\Gamma(B)$ as a $C^\infty(\M)$-module, we may extend $\ka$ to $\Gamma(\D)$ by
\[ \ka(Y) = \ka(f^i) \pb i \]
for any local section $Y=f^i\pb i$ of $\D$, with $f^i$ local functions on $\M$.
This is independent of the chosen local basis since $\ka$ vanishes on $C^\infty(M)\subset C^\infty(\M)$.
Similarly, $\ka$ is extended to $\Gamma(\Hom(\D\otimes\D,\D))$.
Note that $\delta$ could have been defined in the same way, as it commutes with $\Gamma(B)\subset\Gamma(\D)$.


\subsection{Contracting homotopy}
Let $F$ be one of the following three vector bundles over $M$,
\begin{equation}
\label{eq:def-F}
  F=M\times \KK, \qquad
  F=B, \qquad
  F=\Hom(B\otimes B,B) .
\end{equation}
We have graded vector bundles 
$$F_\M=\gm{A[1]\oplus B[1] \oplus B\oplus F}$$ over $\M$ (with sections $C^\infty(\M)$, $\Gamma(\D)$, and $\Gamma(\Hom(\D\otimes\D,\D))$, respectively) and $$F_\A=\gm{A[1]\oplus F}$$ over $\A$ such that $F_\M \cong \pi^* F_\A$.

Consider the space $\Gamma(\Lambda^p A^\lor \otimes F)$ of degree $p$ sections of $F_\A$ as a positive complex concentrated in degree zero with zero differential.
Consider the space $\Omega^{p,\bullet} (\hat S(B^\lor) \otimes F)$ of sections of $F_\M$ with $A$-form degree $p$ as a positive complex graded by the $B$-form degree with $\delta$ as differential.
The pullback $\iota^*$ is a morphism between these two complexes
\[
  \vcenter{\xymatrix{
    \Omega^{p,0} (\hat S(B^\lor) \otimes F)
    \ar[r]^\delta \ar@<-.5ex>[d]_{\iota^*}
    &
    \Omega^{p,1} (\hat S(B^\lor) \otimes F)
    \ar[r]^-\delta \ar[d]_{\iota^*}
    &
    \cdots
    \\
    \Gamma(\Lambda^p A^\lor \otimes F)
    \ar[r]^0 \ar@<-.5ex>[u]_{\pi^*}
    &
    0
    \ar[r]^0
    &
    \cdots
  }} 
\]
and we have the following lemma.

\begin{lemma}\label{hom}
	The map $\ka$, together with $\pi^*:\Gamma(\Lambda^p A^\lor \otimes F)\to\Omega^{p,0} (\hat S(B^\lor) \otimes F)$, is a contracting homotopy of $(\Gamma(\M,F_\M),\delta)$ over $(\Gamma(\A,F_\A),0)$,
	i.e.\ we have
	\begin{enumerate}
	  \item $\ka \circ \ka = 0$,
	  \item $\iota^* \circ \pi^* = \id$,
	  \item \label{homotopy} $\delta\circ \ka + \ka \circ \delta = {\id} - \pi^* \circ \iota^*$ (homotopy formula).  
	\end{enumerate}
\end{lemma}

\begin{proof}
Straightforward computations.
\end{proof}

\begin{corollary}
  For all $p\in\NN$, the map 
  \[ \iota^* : 
    (\Omega^{p,\bullet}(\hat S(B^\lor)\otimes F),\delta)
    \to
    (\Gamma(\Lambda^p A^\lor\otimes F),0)
  \]
  is a quasi-isomorphism, i.e.\ 
  \begin{enumerate}
    \item $H^q \left( \Omega^{p,\bullet}(\hat S(B^\lor)\otimes F), \delta \right)=0$ for $q>0$, and
    \item $H^0 \left( \Omega^{p,\bullet}(\hat S(B^\lor)\otimes F), \delta \right)
      \simeq 
      \Gamma(\Lambda^p A^\lor \otimes F)$.
  \end{enumerate}
\end{corollary}


\subsection{Connections}
\label{Connections}

As some of the forthcoming computations are most easily carried out in coordinates, let us give here coordinate expressions for connections and their associated tensors.
  
We defined linear coordinates $b^i$ on the fibers of $B$ in \eqref{eq:coordinates}, which we can see as local constant sections of $B^\lor$.
Let us write $b_i=\pb{i}$, $\lambda^i$ for the coordinates $\alpha^i$ and $\beta^i$ collectively, and $l_i$ for a dual basis.
  
The Christoffel symbols of an $L$-connection $\nabla$ on $B$ are defined by
\[ \nabla(b_j) = \lambda^i \Gamma_{ij}^k b_k . \]
The operator extending $\nabla$ to $B$-valued $L$-forms reads
\begin{equation*}
	\nabla = d_L + \Gamma
	= \lambda^i \rho_i^j \frac{\partial}{\partial x^j}
	- \frac12 \lambda^i \lambda^j C_{ij}^k \frac{\partial}{\partial \lambda^k}
	+ \lambda^i \Gamma_{ij}^k b_k \frac{\partial}{\partial b_j} ,
\end{equation*}
where $\rho_i^j$ are the components of the anchor map, and $C_{ij}^k$ are the structure functions of $L$.
  
The torsion tensor \eqref{eq:torsion} is
\[ 
  T (l_i,l_j)
  = T_{ij}^k b_k
\]
which yields
\[ T_{ij}^k = \Gamma_{ij}^k - \Gamma_{ji}^k - C_{ij}^k \]
for $i,j,k$ $B$-indices, and
\[ T_{ij}^k = \Gamma_{ij}^k - C_{ij}^k \]
for $i$ an $A$-index, and $j,k$ $B$-indices.

The curvature tensor is
\[ R(l_i,l_j)b_k = R_{ijk}^l b_l , \]
or $R=\frac12\lambda^i \lambda^j R_{ijk}^l b_l \frac{\partial}{\partial b_k}$, which yields
\[ 
  R_{ijk}^l = \rho(\lambda_i)(\Gamma_{jk}^l) - \rho(\lambda_j)(\Gamma_{ik}^l)
  + \Gamma_{im}^l \Gamma_{jk}^m - \Gamma_{jm}^l \Gamma_{ik}^m
  - C_{ij}^m \Gamma_{mk}^l .
\]
An easy computation shows that 
\[ \nabla^2 = R . \]
The Atiyah cocycle $\At^\nabla$ of $\nabla$ is the restriction of $R$ to $l_i\in A$ and $l_j\in B$, see~\eqref{eq:def-atiyah-cocycle-lie-pair}.

The dual $L$-connection on $B^\lor$ has Christoffel symbols $-\Gamma_{ij}^k$,
\[ \nabla(b^k) = - \lambda^i \Gamma_{ij}^k b^j . \]
Its torsion and curvature tensors  $T^\lor$ and $R^\lor$ have coordinates $-T_{ij}^k$ and $-R_{ijk}^l$.
So, for the record, we note
\begin{equation*}
\label{eq:curvature-dual-connection}
  R^\lor = - \frac12\lambda^i\lambda^j R_{ijk}^l b^k \frac{\partial}{\partial b^l} .
\end{equation*}
  
The extension of $\nabla$ to $\hat S(B^\lor)$-valued $L$-forms, i.e.\ to $C^\infty(\M)$, reads
\begin{equation}
\label{eq:connection-on-SBdual-coordinates}
	\nabla 
	= \lambda^i \rho_i^j \frac{\partial}{\partial x^j}
	- \frac12 \lambda^i \lambda^j C_{ij}^k \frac{\partial}{\partial \lambda^k}
	- \lambda^i \Gamma_{ij}^k b^j \frac{\partial}{\partial b^k} 
\end{equation}
and has again curvature $-R_{ijk}^l$.
The derivation $\nabla$ of $C^\infty(\M)$ satisfies
\begin{equation}
\label{eq:connection-on-SBdual-square-is-Rdual}
  \nabla^2 = R^\lor . \qedhere
\end{equation}

\begin{lemma}
\label{lem:torsion-free-connection}
  Let $\nabla$ be a torsion-free $L$-connection on $B$. Then 
  \begin{enumerate}
     \item \label{lem:torsion-free-connection-4} $R(a,b)b' = R(a,b')b$ for $a\in\Gamma(A)$, $b,b'\in\Gamma(B)$.
  \end{enumerate}
  Moreover, its extension $\nabla$ to $\hat S(B^\lor)$ satisfies
  \begin{enumerate}[resume]
    \item \label{lem:torsion-free-connection-1} $\nabla\circ\delta+\delta\circ\nabla = 0$;
    \item \label{lem:torsion-free-connection-2} $\nabla(R^\lor)=0$ (first Bianchi identity);
    \item \label{lem:torsion-free-connection-3} $\delta(R^\lor)=\delta\circ R^\lor + R^\lor\circ\delta=0$ (second Bianchi identity).
    \end{enumerate}
\end{lemma}

\begin{proof}
  The first item is a direct computation.
  The second item also follows from a direct computation using the coordinate expressions above.
  The third item is the usual first Bianchi identity.
  The last item follows from the second one and the identity $\nabla^2 = R^\lor$.
\end{proof}


\subsection{Fedosov differential}

In this section, we build the homological vector field on the graded manifold $\M$ as a Fedosov differential, and use it to put a dg-Lie algebroid structure on the graded vector bundle $\D$ defined in \eqref{eq:def-D}.

All spaces of sections that contain tensor products with sections of $\hat{S}(B^\lor)$ 
are naturally graded by the polynomial degree of the tensor product. 
We will call it the \emph{$b$-degree}.
By construction, $\ka$ raises the $b$-degree by one.

The next lemma, or its idea, is key to many of the proofs here and in all previous papers using Fedosov techniques, starting with \cite{fedosov_simple_1994}.

\begin{lemma}[The Fedosov trick]
\label{lem:key}
  Let $D=-\delta + D_{\geq 0}$ be a derivation of $\Gamma(\Lambda(L^\lor)\otimes \hat S(B^\lor))$ with $D_{\geq 0}$ not decreasing the $b$-degree.
  Let $s$ be a section of $\Lambda(L^\lor)\otimes \hat S(B^\lor) \otimes F$ with $F$ one of the bundles in \eqref{eq:def-F}.
  
  If $D(s)=0$, $\ka(s)=0$, and $\iota^*(s)=0$, then $s=0$.
\end{lemma}

\begin{proof}
  The equation $D(s)=0$ reads
	\begin{equation*}\label{sc2}
	  - \delta (s) + D_{\geq 0}(s) = 0 .
	\end{equation*}
	Applying $\ka$, and using the assumptions $\ka(s)=0$, $\iota^*(s)=0$, and the homotopy formula, one gets
	\begin{equation*}
 		s=\ka(D_{\geq 0}(s)) .
	\end{equation*}
	By iteration, this implies $s=0$ since $\ka$ raises the $b$-degree and $D_{\geq 0}$ does not decrease it.
\end{proof}

\begin{theorem}\label{diffdef}
  Let $\nabla^L$ be a torsion-free $L$-connection on $B$ extending the $A$-action.
  There is a unique vector field
  \begin{equation*}\label{whatisA}
    X=\sum_{k\geq 2}X_k
  \end{equation*}
  with $X_k\in\Gamma(\Lambda^1L^\lor \otimes {S}^{k}(B^\lor)\otimes B)$ satisfying the equation $\ka(X)=0$ and such that the derivation
  $D : \Gamma(\Lambda^\bullet L^\lor \otimes \hat{S}(B^\lor))\to  \Gamma(\Lambda^{\bullet+1} L^\lor \otimes \hat{S}(B^\lor))$
  defined by $$D = \nabla^L - \delta + X$$ satisfies $D^2=0$.
  This solution is such that $X_2=\ka(R^{\lor})$.
\end{theorem}

\begin{proof}
  Let $X=\sum_{k\geq 2}X_k$ be a vector field with $X_k\in\Gamma(\Lambda^1L^\lor \otimes {S}^{k}(B^\lor)\otimes B)$ and such that $\ka(X)=0$. 
  By definition, we have $\iota^*(X)=0$.
  
  The operator $D=\nabla^L-\delta+X$ squares to zero if and only if $\nabla^L+X$ satisfies the Maurer--Cartan equation for $-\delta$, which is equivalent to $C=0$, with
	\begin{equation*}\label{C}
	  C=R^{\lor}+\nabla^L(X)-\delta (X)+\frac{1}{2}[X,X] .
	\end{equation*}
  Here, we used the fact that $\delta(\nabla^L)=0$ and the identity $(\nabla^L)^2=R^\lor$ (see Lemma~\ref{lem:torsion-free-connection}).
  
  Assume that $C=0$.
  Applying $\ka$ to $C=0$ and using $\ka(X)=0$, $\iota^*(X)=0$ and the homotopy formula (Lemma \ref{hom} \eqref{homotopy}), we get the equation
  \begin{equation*}\label{whatA}
    X = \ka \left( R^{\lor}+\nabla^{L}(X)+\frac{1}{2}[X,X] \right)
  \end{equation*} 
  which has a unique solution in $\Gamma(\Lambda^1L^\lor \otimes \hat{S}^{\geq 2}(B^\lor)\otimes B)$. Indeed, by definition the map $\ka$ raises the $b$-degree so, starting with the term $X_2=\ka(R^{\lor})\in \Gamma(\Lambda^1L^\lor\otimes {S}^{2}(B^\lor)\otimes B)$, one may then determine the higher order terms by iteration. 

  With this solution $X$, let us show that $C=0$.
  By definition of $X$, we have $\ka(C)=0$.
  Moreover, for $b$-degree reasons, we have $\iota^*(C)=0$.
  Finally, using items two and three of Lemma~\ref{lem:torsion-free-connection} and the graded Jacobi identity, we compute that
	\begin{equation*}\label{c2}
    D(C) = \nabla^L(C)-\delta (C)+[X,C] = 0 .
	\end{equation*}
  Lemma~\ref{lem:key} then implies that $C=0$.
\end{proof}

\begin{proposition}
\label{prop:M-dg-mfd-such-that-blabla}
  The vector field $D$ turns $\M$ into a dg-manifold such that the inclusion $\iota:\gm{A[1]}\to \M$ and the projection $p:\M\to \gm{L[1]}$ are morphisms of dg-manifolds.
\end{proposition}

\begin{proof}
  Theorem \ref{diffdef} shows that $D$ turns $\M$ into a dg-manifold.
  Since $A$ and $L$ are Lie algebroids, $\gm{A[1]}$ and $\gm{L[1]}$ also have a structure of dg-manifold with the Lie algebroid differential as homological vector field.
  
  Let us show that 
  \[ \iota^* : (\Gamma(\Lambda(L^\lor)\otimes \hat S(B^\lor)), D) \to (\Gamma(\Lambda(A^\lor)), d_A) \]
  is a morphism of complexes. This will imply that $\iota=(\id_M,\iota^*)$ is a morphism of dg-manifolds.
  Since $\iota^*$ is a morphism of algebras, we only need to check the property on generators. In local coordinates \eqref{eq:coordinates}, we have $D(1\otimes b^k) = -\lambda^i \Gamma_{ij}^k b^j - \beta^k$ plus terms of $b$-degree greater than or equal to two. 
  Hence $(\iota^* \circ D)(1\otimes b^k) = 0 = d_A(\iota^*(1\otimes b^k))$.
  Then we have $D(\lambda^k\otimes 1)=-\frac12 \lambda^i\lambda^j C_{ij}^k$. Since $A$ is a subalgebroid of $L$ and $\iota^*(\alpha^i)=\alpha^i$, $\iota^*(\beta^i)=0$, we have $(\iota^*\circ D)(\lambda^k\otimes 1)=d_A (\iota^*(\lambda^i))$.
  Finally, for $f\in C^\infty(M)$, we have $D(f\otimes 1)=\lambda^i \rho_i^j \frac{\partial}{\partial^j}(f)\otimes 1$ so $(\iota^*\circ D)(f\otimes 1)=d_A(f)$.
  
  Similar calculations show that $p$ is also a morphism of dg-manifolds.
\end{proof}

The next lemma also follows from the fact that the projection $p$ is a morphism of dg-manifolds.
\begin{lemma}
\label{lem:D-integ-distrib-stable-LD}
  Let $\D$ be the graded vector bundle defined in \eqref{eq:def-D}.
  Then $\Gamma(\D)$ is stable under the Lie derivative $L_D$.
\end{lemma}

\begin{proof}
  For a section $Y\in\Gamma(\D)$, the Lie derivative $L_DY$ is the bracket $[D,Y]$ of derivations of $C^\infty(\M)$. Since $D=\nabla^L-\delta+X$ and since $\delta$ and $X$ can be regarded as sections of $\D$, we only need to check that $[\nabla^L,Y]\in\Gamma(\D)$.
  Consider the coordinate expression \eqref{eq:connection-on-SBdual-coordinates} of $\nabla^L$.
  The last term is in $\Gamma(\D)$. The first two do not contain variables from $\hat S(B^\lor)$, and the vector fields $\frac{\partial}{\partial x^i}$, $\frac{\partial}{\partial \lambda^j}$, and $\frac{\partial}{\partial b^k}$ commute. So the second claim is proved.
\end{proof}

\begin{proof}[Proof of Theorem \ref{thm:1st-main-theorem}]\footnote{Let us recall that this theorem was independently found by Stiénon--Xu \cite{stienon_fedosov_2016}.}
  The first statement is the content of Proposition~\ref{prop:M-dg-mfd-such-that-blabla}.
  The second statement follows directly from Lemma~\ref{lem:D-integ-distrib-stable-LD} and Lemma~\ref{lem:integrable-distribution-dg-Lie-algebroid}.
\end{proof}

Since it is a dg-Lie algebroid, $\D$ has an Atiyah class $\alpha_\D$, i.e.\ the Atiyah class of the dg-vector bundle $\D$ with respect to the dg-Lie algebroid $\D$ (see Section \ref{sssec:Atiyah-class-dg-Lie-algebroids}).
The Atiyah cocycle $\At^\nabla$ associated to any choice of $\D$-connection $\nabla$ on $\D$ is a section
\[ \At^\nabla \in \Gamma(\M,\Hom(\D \otimes \D, \D)) . \]
The operator $L_D$ on $\Gamma(\D)$ extends to $\Gamma(\Hom(\D \otimes \D,\D))$, and we simply write $D(\cdot)$ for this operator.
The goal of the next section is to show that the pullback by the inclusion $\iota:\A\to \M$ defines a quasi-isomorphism
\[ \iota^* :  (\Gamma(\Hom(\D \otimes \D,\D)), D) \to (\Gamma(\Lambda(A^\lor)\otimes \Hom(B\otimes B,B)), d_A) \]
sending, in cohomology, the Atiyah class $\alpha_\D$ to the Atiyah class $\alpha_{(L,A)}$.


\section{Quasi-isomorphisms}
\label{sec:quasi-isomorphism}

In this section, we prove our second main result, Theorem~\ref{thm:2nd-main-theorem}.
For Sections~\ref{ssec:matched-pairs} and \ref{ssec:fedosov-resolution}, we adapt the techniques of \cite[Section 10]{calaque_lectures_2011}, in which they treat the case of the matched pair $T^{1,0}X \oplus T^{0,1}X$ for $X$ a complex manifold.


\subsection{Matched pairs}
\label{ssec:matched-pairs}

Keeping the setting of Section~\ref{sec:DG-Lie-algebroids-from-Lie-pairs}, assume now  that $L$ comes from a matched pair $(A,B)$. In this case, the splitting $L/A\to L$ is canonical.
 
Since $A$ and $B$ are both Lie subalgebroids, the degree 1 operator $\nabla^L:\Omega^\bullet(\B)\to\Omega^{\bullet+1}(\B)$ splits into \emph{two} parts
\[ \nabla^L = d_A^\nabla + d_B^\nabla \]
of degree $(1,0)$ and $(0,1)$, respectively. The operator $d_A^\nabla$ coincides with $d_A$ on $\Omega^{\bullet,0}(\B)$ and squares to zero, $(d_A^\nabla)^2=0$.

Similarly, we can write 
$$X=\XA +\XB$$ 
with 
\begin{align*}
  \XA &\in \Gamma(\Lambda^1A^\lor\otimes\hat{S}^{\geq2}(B^\lor)\otimes B), \\
  \XB &\in \Gamma(\Lambda^1B^\lor\otimes\hat{S}^{\geq2}(B^\lor)\otimes B).
\end{align*}

With these decompositions, we have
\[ D= \DA + \DB  \]
with, for all $p,q\geq 0$,
\begin{equation*}\label{d2}
\begin{split}
  &\DA  : \Omega^{p,q}(\B)\to \Omega^{p+1,q}(\B) 
  , \qquad \DA  = d_A^\nabla+[\XA ,\cdot],
\end{split}
\end{equation*}
and
\begin{equation*}\label{d1}
\begin{split}
  &\DB  : \Omega^{p,q}(\B)\to \Omega^{p,q+1}(\B) 
  , \qquad \DB  = d^\nabla_B-\delta+[\XB ,\cdot] .
\end{split}
\end{equation*}
Checking the degrees  of the image of $D^2$, it is immediate that $D^2=0$ implies
\begin{equation}\label{ds}
  (\DA )^2=0 , \qquad \DA \circ \DB +\DB \circ \DA =0 , \qquad (\DB )^2=0 .
\end{equation}


\subsection{Fedosov resolution}
\label{ssec:fedosov-resolution}

The goal of this section is to prove that there is a resolution
\[ 
  0 
  \to 
  \Gamma(\Lambda^\bullet A^\lor \otimes \Hom(B\otimes B, B))
  \xrightarrow{\eta\circ\mu\circ\pi^*}
  \Omega^{\bullet,0}(\B)
  \xrightarrow{\DB }
  \Omega^{\bullet,1}(\B)
  \xrightarrow{\DB }
  \Omega^{\bullet,2}(\B)
  \xrightarrow{\DB }
  \cdots
\]
of the complex 
\begin{equation}
\label{eq:small-complex}
  (\Gamma(\Lambda^\bullet A^\lor \otimes \Hom(B\otimes B, B)),d_A)
\end{equation}
by the complexes
\[ (\Omega^{\bullet,q}(\B), d_A^\nabla) , \qquad q \geq 0 . \]
This will imply that the inclusion $\iota:\A\to\M$ induces a quasi-isomorphism between the total complex $(\Omega^{\bullet}(\B), D)$ and the complex \eqref{eq:small-complex} computing the $\Hom(B\otimes B, B)$-valued $A$-cohomology.

\begin{remark}
  In the case of a general Lie pair, the operator $D$ splits into \emph{three} parts $D^{1,0}$, $D^{0,1}$, and $D^{-1,2}$ of bidegree $(1,0)$, $(0,1)$, and $(-1,2)$, respectively.
  We still have $(D^{1,0})^2 = 0$ but the square of $D^{0,1}$ no longer vanishes. Hence $(\Omega^{\bullet,\bullet}(\B),D^{1,0},D^{0,1})$ is no longer a bicomplex and we cannot apply the spectral sequence argument of Proposition \ref{cohos} to build the quasi-isomorphism alluded to in the foregoing paragraph.
\end{remark}

Since we will use the projection $\pi^* \circ \iota^*$ many times, let us set
\begin{equation}
\label{eq:definition-sigma}
  \sigma = \pi^* \circ \iota^* .
\end{equation}

\begin{lemma}
\label{lem:iota-star-morphism-of-complexes}
  The map
  \begin{equation}
  \label{eq:iota-star-from-big-complex}
    \iota^* : 
    (\Omega^{\bullet}(\B),D)
    \to
    (\Gamma(\Lambda^\bullet A^\lor\otimes \Hom(B\otimes B,B)),d_A)
  \end{equation}
  is a morphism of complexes.
\end{lemma}

\begin{proof}
  Since $\DB$ raises the $B$-form degree, we have $\iota^* \circ \DB=0$, so it is enough to prove 
  \begin{equation*}
  \label{eq:sigma-D-dA-sigma}
    \iota^* \circ \DA  = d_A \circ \iota^* .
  \end{equation*}
  Since $[\XA ,\cdot]$ raises the $b$-degree, we get
  \[ \iota^* \circ \DA  = \iota^*  \circ (d_A^\nabla + [\XA , \cdot]) = d_A \circ \iota^* . \qedhere \]
\end{proof}

The next few results will be devoted to building an explicit quasi-inverse to the map $\iota^*$ in \eqref{eq:iota-star-from-big-complex}. 

\begin{proposition}\label{cohos}
  The inclusion map $\eta$ from $\Omega^{\bullet,0}(\B)\cap\ker(\DB )$ to $\Omega^\bullet(\B)$ induces an isomorphism
  \begin{equation}\label{first}
    H^\bullet(\Omega^{\bullet,0}(\B)\cap\ker(\DB ),\DA ) 
    \to
    H^\bullet\left( \Omega^\bullet(\B), D\right) .
  \end{equation}
\end{proposition}

\begin{proof}
  Consider the double complex
  \[C^{\bullet,\bullet}=(\Omega^{\bullet,\bullet}(\B),\DA ,\DB ) . \]
  The right-hand side of \eqref{first} is the cohomology of the total complex $C^\bullet = (\Omega^\bullet(\B),D)$.

  We first want to compute the $\DB $-cohomology of $C^{\bullet,\bullet}$.

  Let $a$ be an element of $\Omega^{\bullet,q}(\B)$ for $q>0$ such that $\DB (a)=0$.
  We will show that the equation $a=\DB (b)$ has an explicit solution $b$.
  Since $\ka$ raises the $b$-degree, solving the recursive equation
  \begin{equation}
  \label{eq:b-recursion}
    b=\ka(-a + d^\nabla_B(b) + [\XB,b])
  \end{equation}
  yields a unique element $b\in \Omega^{\bullet,q-1}(\B)$.
  Applying $\delta$ to \eqref{eq:b-recursion} and using the homotopy formula gives
  \[ \DB (b)-a=\kappa \delta (-a + d^\nabla_B (b) + [\XB,b]) . \]
  Writing $h=\DB (b)-a\in \Omega^{\bullet,q}(\B)$, we see that $\DB (h)=0$, $\ka(h)=0$, and $\iota^*(h)=0$.
  Lemma \ref{lem:key} then shows that $h=0$.
  
  We just proved that the $\DB $-cohomology of $C^{\bullet,\bullet}$ is
  \[ H^\bullet(C^{\bullet,\bullet},\DB )=H^0(C^{\bullet,\bullet},\DB )=\Omega^{\bullet,0}(\B)\cap\ker(\DB ).\]

  Let us now consider the spectral sequence associated to the total complex $C^{\bullet}$ and the filtration $F^kC^n = \bigoplus_{p+q=n,p\geq k} C^{p,q}$.
  We have $E_0^{p,q}\cong C^{p,q}$ with $D_0=D_B$.
  Since the $\DB $-cohomology of $C^{\bullet,\bullet}$ is concentrated in degree 0, the spectral sequence stabilizes at the second page, and the cohomology of the total complex $C^{\bullet}$ is
  \[ H^\bullet(C^{\bullet},D)
  = H^\bullet (H^0(C^{\bullet,\bullet},\DB ),\DA )
  = H^\bullet (\Omega^{\bullet,0}(\B)\cap\ker(\DB ),\DA ) . \qedhere\]
\end{proof}

\begin{proposition}\label{mu}
  The formula 
  \begin{equation}\label{r}
    \mu(a) = a + \ka((\DB +\delta)(\mu(a)))
  \end{equation}
  defines an isomorphism of graded vector spaces
  \[
    \mu:
    \Omega^{\bullet,0}(\B)\cap \ker(\delta)
    \to
    \Omega^{\bullet,0}(\B)\cap \ker(\DB ) .
  \]
\end{proposition}

\begin{proof}
  Let $a\in\Omega^{\bullet,0}(\B)\cap \ker(\delta)$. 
  We need to show first that the recursion relation \eqref{r} actually defines a unique element $\mu(a)$ and then that
  \begin{equation*}\label{well}
    \DB (\mu(a))=0.
  \end{equation*}
  
  As before, observe that $\kappa$ raises the $b$-degree and the operator $\DB +\delta=d^\nabla_B+[\XB ,\cdot]$ preserves it.
  Consequently, equation \eqref{r} has for each $a$ a unique solution $\mu(a)$ such that
  \begin{equation}
  \label{eq:sigma-left-inverse-of-mu}
    \sigma(\mu(a))=a .
  \end{equation}

  Let us show that $\DB (\mu( a))=0$.
  Notice that $\ka$ vanishes on the image of $\mu$.
  Using the definition of $\mu$ and the homotopy formula one has that
  \begin{equation*}\label{dyo}
    \begin{split}
      \ka(\DB (\mu( a)))
      & = \ka\left((\DB +\delta)(\mu( a))\right)-\ka\delta(\mu( a)) \\
      & = \mu( a)- a + \delta \ka(\mu( a))-\mu( a)+\sigma(\mu( a)) \\
      & = 0 .
    \end{split}
  \end{equation*}
  We have $\DB (\DB (\mu(a)))=0$
  and $\DB (\mu( a))\in\Omega^{\bullet,1}(\B)\subset\ker(\iota^*)$.
  It follows from Lemma \ref{lem:key} that $\DB (\mu( a))=0$.
  
  Let us show that $\mu$ is an isomorphism.
  Since $\mu$ has left-inverse $\sigma$ (see \eqref{eq:sigma-left-inverse-of-mu}), it is injective.
  To prove that $\mu$ is surjective, observe that for $b\in\Omega^{\bullet,0}(\B)\cap\ker(\DB )$ one has $\sigma(\mu(\sigma(b)))=\sigma(b)$.
  But then $b=\mu(\sigma(b))$ by Lemma \ref{lem:key}.
\end{proof}

\begin{proposition}\label{theo}
  The map $\mu$ is an isomorphism of complexes
  \begin{equation*}\label{inv2}
    \mu:
    (\Omega^{\bullet,0}(\B)\cap \ker(\delta),d_A^\nabla)
    \to
    (\Omega^{\bullet,0}(\B)\cap \ker(\DB ),\DA ) .
  \end{equation*}
\end{proposition}

\begin{proof}
  Let $a\in\Omega^{\bullet,0}(\B)\cap\ker(\delta)$.
  We will show that $\DA (\mu(a))=\mu(d_A^\nabla(a))$ by applying Lemma \ref{lem:key} to their difference.
  
  For $B$-form degree reasons, $\ka$ vanishes on both elements.
  By definition of $\mu$, $\DB\circ\mu=0$ so $\DB (\mu(d_A^\nabla(a)))=0$.
  By \eqref{ds},
  \[\DB (\DA (\mu(a)))=-\DA (\DB (\mu(a)))=0.\]
  Furthermore, since $[\XA ,\operatorname{\cdot}{}]$ raises the $b$-degree, we have
  \begin{equation*}
    \begin{split}
      \sigma(\DA (\mu(a)))
      & = \sigma\circ{\sgn}\left(d_A^\nabla(\mu(a))+[\XA ,\mu(a)]\right)
      = \sigma(d_A^\nabla(\mu((a))) \\
      & = \sigma\left(d_A^\nabla(a)+d_A^\nabla(k(\DB +\delta)(\mu(a)))\right) \\
      &= \sigma(d_A^\nabla(a))
      = d_A^\nabla(a)
      = \sigma(\mu(d_A^\nabla(a))).
    \end{split}
  \end{equation*}
  So Lemma \ref{lem:key} gives $\DA (\mu(a))=\mu(d_A^\nabla(a))$.
\end{proof}

\begin{proposition}\label{prop:iota-star-quasi-isomorphism}
  The map \eqref{eq:iota-star-from-big-complex} is a quasi-isomorphism.
\end{proposition}

\begin{proof}
  Consider the map $\eta\circ\mu\circ\pi^*$ (see Propositions~\ref{cohos} and \ref{theo}).
  On the one hand, it is a right-inverse of $\iota^*$. 
  Indeed, applying $\iota^*$ on the left and $\pi^*$ on the right of \eqref{eq:sigma-left-inverse-of-mu} yields $\iota^*\circ(\eta\circ\mu\circ\pi^*)=\id$.
  On the other hand, we have seen in Propositions~\ref{cohos} and \ref{theo} that $\eta$ and $\mu$ are quasi-isomorphisms, and by definition of $\pi$ and $d_A^\nabla$, $\pi^*$ is an isomorphism of complexes onto the domain of $\mu$. 
  As a result, $\iota^*$ is also a quasi-isomorphism, with quasi-inverse $\eta\circ\mu\circ\pi^*$.
\end{proof}


\subsection{Atiyah classes}
\label{ssec:Atiyah-classes}

We begin with a short lemma about $\D$-connections on $\D$.

\begin{lemma}
\label{lem:connection-is-bilinear}
  Let $\nabla$ be a $\D$-connection on $\D$.
  Then there exists a unique $S \in \Gamma(\Hom(\iota^*\D \otimes \iota^*\D, \iota^*\D)) \cong \Gamma(\Lambda(A^\lor) \otimes \Hom(B\otimes B,B))$ such that
  \[ \iota^*(\nabla_XY) = S(\iota^*X,\iota^*Y) \]
  for all $X,Y\in\Gamma(\D)$ of $b$-degree zero.
\end{lemma}
\begin{proof}
  There is an explicit $\D$-connection $\nabla^0$ on $\D$ defined on constant sections
  \begin{equation}
  \label{eq:constant-sections-of-D}
    \partial_i 
    := 1 \otimes 1 \otimes 1 \otimes \frac{\partial}{\partial b^i} 
    \in \Gamma(\Lambda(A^\lor) \otimes \Lambda(B^\lor) \otimes \hat S(B^\lor) \otimes B)
  \end{equation}  
  by
  \[ \nabla^0_{\partial_i} \partial_j = 0 \]
  and extended by linearity and the Leibniz rule.
  It is easy to check that $\nabla^0$ is well-defined and satisfies $\iota^*(\nabla^0_XY)=0$ for all sections $X,Y\in\Gamma(\D)$ of $b$-degree zero.
  
  Now any $\D$-connection on $\D$ is of the form $\nabla = \nabla^0+T$ for some $T\in\Gamma(\Hom(\D\otimes\D,\D))$.
  The result follows with $S=\iota^*(T)$.
\end{proof}

We can now turn to the proof of our second main theorem.

\begin{proof}[Proof of Theorem \ref{thm:2nd-main-theorem}]
  The fact that $\iota^*$ is a quasi-isomorphism was proved in Proposition~\ref{prop:iota-star-quasi-isomorphism}.
  
  Let us now prove the statement about the Atiyah class.
  Recall that $\iota$ is the inclusion of $\A$ in $\M$ and that the bundle $\gm{A[1]\oplus B}\to\gm{A[1]}$ is isomorphic to the pullback bundle of $\D\to\M$ by $\iota$.
  Hence, we have pullback maps
  \[ \iota^* : \Gamma(\D) \to \Gamma(\Lambda(A^\lor) \otimes B) \]
  and
  \[ \iota^* : \Gamma(\Hom(\D \otimes \D, \D)) \to \Gamma(\Lambda(A^\lor) \otimes \Hom(B \otimes B, B)) . \]
  
  For all $X,Y\in\Gamma(\D)$, there is a commuting square
  \begin{equation*}
    \xymatrixcolsep{6pc}
    \vcenter{\xymatrix{
      \Gamma(\Hom(\D \otimes \D, \D)) \ar[r]^-{\ev_{X,Y}} \ar[d]_{\iota^*}
      & \Gamma(\D) \ar[d]^{\iota^*} \\
      \Gamma(\Lambda (A^\lor) \otimes \Hom(B \otimes B, B)) \ar[r]_-{\ev_{\iota^*(X),\iota^*(Y)}}
      & \Gamma(\Lambda (A^\lor) \otimes B)
    }} 
  \end{equation*}
  where $\ev$ denotes the evaluation.
  
  Let $\nabla$ be a $\D$-connection on $\D$.
  Recall the definitions of the Atiyah cocycles from \eqref{eq:def-atiyah-cocycle-lie-pair} and \eqref{eq:definition-Atiyah-cocycle-dgLiealgebroid}.
  We are going to check that
  \begin{equation}
  \label{eq:atiyahs-play-together}
    \iota^*(\At^\nabla(X,Y)) = \At^{\nabla^L}(\iota^*X,\iota^*Y) + d_A(S)(\iota^*X,\iota^*Y),
  \end{equation}
  for all $X,Y\in\Gamma(\D)$ and for some $S\in\Gamma(\Hom(\iota^*\D \otimes \iota^*\D, \iota^*\D))$.
  This will imply that $\iota^*(\alpha_\D)=\alpha_{(L,A)}$ in cohomology.
    
  Since the Atiyah cocycle is $C^\infty(\M)$-linear, we may restrict to constant sections \eqref{eq:constant-sections-of-D}. 
  On such elements, we have
  \begin{equation*}\label{eq:atiyah-on-partial-ij}
    \At^\nabla(\partial_i, \partial_j)
    = D \nabla_{\partial_i} \partial_j - \nabla_{D\partial_i}\partial_j - \nabla_{\partial_i} D\partial_j .
  \end{equation*}
  Recall that $D=\nabla^L-\delta+[X,\cdot]$. We will now consider the contribution of each term of $D$ separately.
  
  Since $\delta$ lowers the $b$-degree, $\delta(\partial_i)=0$.
  Moreover, since $\delta$ raises the $B$-form degree, $\delta(\nabla_{\partial_i}\partial_j)$ is of $B$-form degree at least one (actually, one), so $\iota^*$ vanishes on it.
  Hence, the terms involving $\delta$ in $\iota^*(\At^\nabla(\partial_i,\partial_j))$ vanish.
  
  Consider now the terms involving $\nabla^L$.
  Recall that $\nabla^L = d_A^\nabla + d_B^\nabla$.
  Since $\nabla$ is linear in the first argument and $d_B^\nabla$ raises the $B$-form degree by one, we obviously have
  \[ \iota^*(d_B^\nabla(\nabla_{\partial_i}\partial_j)) = 0 \quad \text{and} \quad \iota^*(\nabla_{d_B^\nabla(\partial_i)}\partial_j) = 0 . \]
  The term $\iota^*(\nabla_{\partial_i}d_B^\nabla(\partial_j))$ vanishes for the same reason using the Leibniz identity
  \[ \nabla_X(fY) = X(f) Y + (-1)^{|f||X|} f \nabla_XY \]
  where $X,Y\in\Gamma(\D)$ and $f\in C^\infty(\M)$, and the fact that sections of $\D$ only act on the $\hat S(B^\lor)$ part of $f$ (remember that $\D$ is a subalgebroid of $T\M$, so its anchor map is the inclusion).
  As a result, we get
  \begin{align*}
    \iota^* \Big(
    \nabla^L \nabla_{\partial_i}\partial_j
    - \nabla_{\nabla^L \partial_i} \partial_j
    - \nabla_{\partial_i} \nabla^L\partial_j
    \Big)
    &=
    \iota^* \Big(
    d_A^\nabla \nabla_{\partial_i}\partial_j
    - \nabla_{d_A^\nabla \partial_i} \partial_j
    - \nabla_{\partial_i} d_A^\nabla\partial_j
    \Big) \\
    &= d_A(S)(\iota^*\partial_i,\iota^*\partial_j)
  \end{align*}
  for some $S$ given by Lemma~\ref{lem:connection-is-bilinear}.
  
  Let us finally turn to the elements involving $X$.
  Recall that $X=\sum_{k\geq 2} X_k$ with $X_k\in\Omega^{\bullet,\bullet}(\B^{k})$.
  Consequently, since $\iota^*$ vanishes on elements with non-zero $b$-degree, we have
  \[ 
    \iota^*([X,\nabla_{\partial_i}\partial_j]) = 0 
    \quad \text{and} \quad 
    \iota^*(\nabla_{[X,\partial_i]}\partial_j) = 0 . 
  \]
  Now notice that $[X,\partial_j]$ is equal to $[X_2,\partial_j]$ plus terms of $b$-degree greater than 2 that vanish under $\iota^*$.
  We may decompose $X_2$ as 
  \[ X_{2,A}+X_{2,B} \in \Omega^{1,0}(\B^2) \oplus \Omega^{0,1}(\B^2) \]
  and notice that $\iota^*(\nabla_{\partial_i}[X_{2,B},\partial_j])$ vanishes for $B$-form degree reasons.
  Recall from Theorem~\ref{diffdef} that $X_2=\ka(R^{\lor})$.
  So $X_{2,A}$ takes the explicit form
  \[ X_{2,A} = \frac12 R_{kij}^l \alpha^k \otimes 1 \otimes b^ib^j \otimes \frac{\partial}{\partial b^l} . \]
  Hence, using Lemma~\ref{lem:torsion-free-connection}~(\ref{lem:torsion-free-connection-4}), we get
  \begin{align*}
    \iota^* (\nabla_{\partial_i} [X,\partial_j])
    &= -R_{kij}^l \alpha^k \otimes \frac{\partial}{\partial b^l} 
    = -\At^{\nabla^L}(\partial_i,\partial_j) .
  \end{align*}
  
  Putting everything together, we have obtained \eqref{eq:atiyahs-play-together}.
\end{proof}


\bibliographystyle{amsabbrvlinks}
\bibliography{Atiyah}

\providecommand{\bysame}{\leavevmode\hbox to3em{\hrulefill}\thinspace}
\def\checkMR #1 #2\relax{#1}
\providecommand{\MR}{}\renewcommand{\MR}[1]{\relax\ifhmode\unskip\spacefactor3000
  \space\fi \textsc{mr:}\MRhref{\checkMR #1 \relax}{\checkMR #1 \relax}}
\providecommand{\MRhref}{}\renewcommand{\MRhref}[2]{%
  \href{http://www.ams.org/mathscinet-getitem?mr=#1}{#2}
}
\providecommand{\href}[2]{#2}
\begin{thebibliography}{10}

\bibitem{atiyah_complex_1957}
M.~F. Atiyah, \emph{Complex analytic connections in fibre bundles}, Trans.
  Amer. Math. Soc. \textbf{85} (1957), 181--207.
  {\catcode`\_\active\let_\_\textsc{doi:}\href{http://dx.doi.org/10.1090/S0002-9947-1957-0086359-5}{10.1090/S0002-9947-1957-0086359-5}}
  \MR{0086359 (19,172c)}

\bibitem{calaque_formality_2007}
D.~Calaque, V.~Dolgushev, and G.~Halbout, \emph{Formality theorems for
  {H}ochschild chains in the {L}ie algebroid setting}, J. Reine Angew. Math.
  \textbf{612} (2007), 81--127.
  {\catcode`\_\active\let_\_\textsc{doi:}\href{http://dx.doi.org/10.1515/CRELLE.2007.085}{10.1515/CRELLE.2007.085}}
  \MR{2364075 (2008j:53153)}
  \textsc{arxiv:}\href{http://arxiv.org/abs/math/0504372}{math/0504372}

\bibitem{calaque_lectures_2011}
D.~Calaque and C.~A. Rossi, \emph{Lectures on {D}uflo isomorphisms in {L}ie
  algebra and complex geometry}, EMS Series of Lectures in Mathematics,
  European Mathematical Society (EMS), Z\"urich, 2011.
  {\catcode`\_\active\let_\_\textsc{doi:}\href{http://dx.doi.org/10.4171/096}{10.4171/096}}
  \MR{2816610 (2012j:53121)}

\bibitem{calaque_hochschild_2010}
D.~Calaque and M.~Van~den Bergh, \emph{Hochschild cohomology and {A}tiyah
  classes}, Adv. Math. \textbf{224} (2010), no.~5, 1839--1889.
  {\catcode`\_\active\let_\_\textsc{doi:}\href{http://dx.doi.org/10.1016/j.aim.2010.01.012}{10.1016/j.aim.2010.01.012}}
  \MR{2646112 (2011i:14037)}
  \textsc{arxiv:}\href{http://arxiv.org/abs/0708.2725}{0708.2725}

\bibitem{cattaneo_relative_2007}
A.~S. Cattaneo and G.~Felder, \emph{Relative formality theorem and quantisation
  of coisotropic submanifolds}, Adv. Math. \textbf{208} (2007), no.~2,
  521--548.
  {\catcode`\_\active\let_\_\textsc{doi:}\href{http://dx.doi.org/10.1016/j.aim.2006.03.010}{10.1016/j.aim.2006.03.010}}
  \MR{2304327 (2008b:53119)}
  \textsc{arxiv:}\href{http://arxiv.org/abs/math/0501540}{math/0501540}

\bibitem{chen_atiyah_2012}
Z.~Chen, M.~Sti{\'e}non, and P.~Xu, \emph{From {A}tiyah {C}lasses to {H}omotopy
  {L}eibniz {A}lgebras}, Comm. Math. Phys. \textbf{341} (2016), no.~1,
  309--349.
  {\catcode`\_\active\let_\_\textsc{doi:}\href{http://dx.doi.org/10.1007/s00220-015-2494-6}{10.1007/s00220-015-2494-6}}
  \MR{3439229} \textsc{arxiv:}\href{http://arxiv.org/abs/1204.1075}{1204.1075}

\bibitem{ciocan-fontanine_derived_2001}
I.~Ciocan-Fontanine and M.~Kapranov, \emph{Derived {Q}uot schemes}, Ann. Sci.
  \'Ecole Norm. Sup. (4) \textbf{34} (2001), no.~3, 403--440.
  {\catcode`\_\active\let_\_\textsc{doi:}\href{http://dx.doi.org/10.1016/S0012-9593(01)01064-3}{10.1016/S0012-9593(01)01064-3}}
  \MR{1839580 (2002k:14003)}
  \textsc{arxiv:}\href{http://arxiv.org/abs/math/9905174}{math/9905174}

\bibitem{costello_geometric_2011}
K.~J. {Costello}, \emph{{A geometric construction of the Witten genus, II}},
  2011. \textsc{arxiv:}\href{http://arxiv.org/abs/1112.0816}{1112.0816}

\bibitem{covolo_z2n_2014}
T.~{Covolo}, J.~{Grabowski}, and N.~{Poncin},
  \emph{$\mathbb{Z}_2^n$-supergeometry {I}: manifolds and morphisms}, 2014.
  \textsc{arxiv:}\href{http://arxiv.org/abs/1408.2755}{1408.2755}

\bibitem{dolgushev_covariant_2005}
V.~Dolgushev, \emph{Covariant and equivariant formality theorems}, Adv. Math.
  \textbf{191} (2005), no.~1, 147--177.
  {\catcode`\_\active\let_\_\textsc{doi:}\href{http://dx.doi.org/10.1016/j.aim.2004.02.001}{10.1016/j.aim.2004.02.001}}
  \MR{2102846 (2006c:53101)}
  \textsc{arxiv:}\href{http://arxiv.org/abs/math/0307212}{math/0307212}

\bibitem{fedosov_simple_1994}
B.~V. Fedosov, \emph{A simple geometrical construction of deformation
  quantization}, J. Differential Geom. \textbf{40} (1994), no.~2, 213--238.
  \MR{1293654 (95h:58062)}

\bibitem{gwilliam_one-dimensional_2014}
O.~Gwilliam and R.~Grady, \emph{One-dimensional {C}hern-{S}imons theory and the
  {$\hat A$} genus}, Algebr. Geom. Topol. \textbf{14} (2014), no.~4,
  2299--2377.
  {\catcode`\_\active\let_\_\textsc{doi:}\href{http://dx.doi.org/10.2140/agt.2014.14.2299}{10.2140/agt.2014.14.2299}}
  \MR{3331615} \textsc{arxiv:}\href{http://arxiv.org/abs/1110.3533}{1110.3533}

\bibitem{kapranov_rozansky-witten_1999}
M.~Kapranov, \emph{Rozansky--{W}itten invariants via {A}tiyah classes},
  Compositio Math. \textbf{115} (1999), no.~1, 71--113.
  {\catcode`\_\active\let_\_\textsc{doi:}\href{http://dx.doi.org/10.1023/A:1000664527238}{10.1023/A:1000664527238}}
  \MR{1671737 (2000h:57056)}
  \textsc{arxiv:}\href{http://arxiv.org/abs/alg-geom/9704009}{alg-geom/9704009}

\bibitem{kontsevich_rozansky-witten_1999}
M.~Kontsevich, \emph{Rozansky-{W}itten invariants via formal geometry},
  Compositio Math. \textbf{115} (1999), no.~1, 115--127.
  {\catcode`\_\active\let_\_\textsc{doi:}\href{http://dx.doi.org/10.1023/A:1000619911308}{10.1023/A:1000619911308}}
  \MR{1671725 (2000h:57057)}
  \textsc{arxiv:}\href{http://arxiv.org/abs/dg-ga/9704009}{dg-ga/9704009}

\bibitem{kontsevich_deformation_2003}
\bysame, \emph{Deformation quantization of {P}oisson manifolds}, Lett. Math.
  Phys. \textbf{66} (2003), no.~3, 157--216.
  {\catcode`\_\active\let_\_\textsc{doi:}\href{http://dx.doi.org/10.1023/B:MATH.0000027508.00421.bf}{10.1023/B:MATH.0000027508.00421.bf}}
  \MR{2062626 (2005i:53122)}
  \textsc{arxiv:}\href{http://arxiv.org/abs/q-alg/9709040}{q-alg/9709040}

\bibitem{mehta_algebroids_2009}
R.~A. Mehta, \emph{{$Q$}-algebroids and their cohomology}, J. Symplectic Geom.
  \textbf{7} (2009), no.~3, 263--293. \MR{2534186 (2011b:58040)}
  \textsc{arxiv:}\href{http://arxiv.org/abs/math/0703234}{math/0703234}

\bibitem{mehta_groupoids_2009}
\bysame, \emph{{$Q$}-groupoids and their cohomology}, Pacific J. Math.
  \textbf{242} (2009), no.~2, 311--332.
  {\catcode`\_\active\let_\_\textsc{doi:}\href{http://dx.doi.org/10.2140/pjm.2009.242.311}{10.2140/pjm.2009.242.311}}
  \MR{2546715 (2010j:22005)}
  \textsc{arxiv:}\href{http://arxiv.org/abs/math/0611924}{math/0611924}

\bibitem{mehta_atiyah_2015}
R.~A. Mehta, M.~Sti{\'e}non, and P.~Xu, \emph{The {A}tiyah class of a dg-vector
  bundle}, C. R. Math. Acad. Sci. Paris \textbf{353} (2015), no.~4, 357--362.
  {\catcode`\_\active\let_\_\textsc{doi:}\href{http://dx.doi.org/10.1016/j.crma.2015.01.019}{10.1016/j.crma.2015.01.019}}
  \MR{3319134}
  \textsc{arxiv:}\href{http://arxiv.org/abs/1502.03119}{1502.03119}

\bibitem{mokri_matched_1997}
T.~Mokri, \emph{Matched pairs of {L}ie algebroids}, Glasgow Math. J.
  \textbf{39} (1997), no.~2, 167--181.
  {\catcode`\_\active\let_\_\textsc{doi:}\href{http://dx.doi.org/10.1017/S0017089500032055}{10.1017/S0017089500032055}}
  \MR{1460632 (99a:58165)}

\bibitem{roytenberg_structure_2002}
D.~Roytenberg, \emph{On the structure of graded symplectic supermanifolds and
  {C}ourant algebroids}, Quantization, {P}oisson brackets and beyond
  ({M}anchester, 2001), Contemp. Math., vol. 315, Amer. Math. Soc., Providence,
  RI, 2002, pp.~169--185.
  {\catcode`\_\active\let_\_\textsc{doi:}\href{http://dx.doi.org/10.1090/conm/315/05479}{10.1090/conm/315/05479}}
  \MR{1958835 (2004i:53116)}
  \textsc{arxiv:}\href{http://arxiv.org/abs/math/0203110}{math/0203110}

\bibitem{severa_some_2005}
P.~{\v{S}}evera, \emph{Some title containing the words ``homotopy'' and
  ``symplectic'', e.g. this one}, Travaux math\'ematiques. {F}asc. {XVI}, Trav.
  Math., XVI, Univ. Luxemb., Luxembourg, 2005, pp.~121--137. \MR{2223155
  (2007f:53107)}
  \textsc{arxiv:}\href{http://arxiv.org/abs/math/0105080}{math/0105080}

\bibitem{shoikhet_duflo_1998}
B.~{Shoikhet}, \emph{{On the Duflo formula for $L_\infty$-algebras and
  Q-manifolds}}, 1998.
  \textsc{arxiv:}\href{http://arxiv.org/abs/math/9812009}{math/9812009}

\bibitem{stienon_fedosov_2016}
M.~Stiénon and P.~Xu, \emph{Fedosov dg manifolds and {G}erstenhaber algebras
  associated with {L}ie pairs}, 2016.
  \textsc{arxiv:}\href{http://arxiv.org/abs/1605.09656}{1605.09656}

\bibitem{voronov_graded_2002}
T.~Voronov, \emph{Graded manifolds and {D}rinfeld doubles for {L}ie
  bialgebroids}, Quantization, {P}oisson brackets and beyond ({M}anchester,
  2001), Contemp. Math., vol. 315, Amer. Math. Soc., Providence, RI, 2002,
  pp.~131--168.
  {\catcode`\_\active\let_\_\textsc{doi:}\href{http://dx.doi.org/10.1090/conm/315/05478}{10.1090/conm/315/05478}}
  \MR{1958834 (2004f:53098)}
  \textsc{arxiv:}\href{http://arxiv.org/abs/math/0105237}{math/0105237}

\end{thebibliography}

\end{document}